\documentclass{gtart}
\pdfoutput=1
\newcommand{\pathtotrunk}{./}
\usepackage{ifpdf}

\ifpdf
\usepackage[pdftex,plainpages=false,hypertexnames=false,pdfpagelabels]{hyperref}
\else
\usepackage[dvips,plainpages=false,hypertexnames=false]{hyperref}
\fi

\usepackage[usenames]{color}

\usepackage{amsmath,amssymb,amsfonts}
\usepackage{enumerate}

\usepackage{leftidx,soul,ulem}

\ifpdf
\newcommand{\pathtodiagrams}{\pathtotrunk diagrams/pdf/}
\else
\newcommand{\pathtodiagrams}{\pathtotrunk diagrams/eps/}
\fi

\newcommand{\mathfig}[2]{\ensuremath{{\hspace{-3pt}\begin{array}{c}%
  \raisebox{-2.5pt}{\includegraphics[width=#1\textwidth]{\pathtodiagrams #2}}%
\end{array}\hspace{-3pt}}}}

\newcommand{\arxiv}[1]{\href{http://arxiv.org/abs/#1}{\tt arXiv:\nolinkurl{#1}}}
\newcommand{\doi}[1]{\href{http://dx.doi.org/#1}{{\tt DOI:#1}}}

\newcommand{\googlebooks}[1]{(preview at \href{http://books.google.com/books?id=#1}{google books})}

\theoremstyle{plain}
\newtheorem{prop}{Proposition}[section]

\newtheorem{thm}[prop]{Theorem}

\newtheorem{lem}[prop]{Lemma}

\newtheorem*{cor*}{Corollary}

\newtheorem{defn}[prop]{Definition}         % numbered definition
\newtheorem*{defn*}{Definition}             % unnumbered definition
\newtheorem*{notation*}{Notation} 

\newtheorem{example}{Example}

\newenvironment{rem}{\noindent\textsl{Remark.}}{}  % perhaps looks better than rem above?
\numberwithin{equation}{section}
\newcounter{comment}
\newcommand{\noop}[1]{}

\def\clap#1{\hbox to 0pt{\hss#1\hss}}

\newcommand{\Natural}{\mathbb N}
\newcommand{\Integer}{\mathbb Z}
\newcommand{\Rational}{\mathbb Q}
\newcommand{\Real}{\mathbb R}
\newcommand{\Complex}{\mathbb C}

\newcommand{\cC}{\mathcal{C}}
\newcommand{\cD}{\mathcal{D}}

\newcommand{\cF}{\mathcal{F}}
\newcommand{\cG}{\mathcal{G}}
\newcommand{\cH}{\mathcal{H}}

\newcommand{\cK}{\mathcal{K}}

\newcommand{\cM}{\mathcal{M}}
\newcommand{\cN}{\mathcal{N}}

\newcommand{\cP}{\mathcal{P}}
\newcommand{\cQ}{\mathcal{Q}}

\newcommand{\cS}{\mathcal{S}}

\newcommand{\id}{\boldsymbol{1}}
\renewcommand{\imath}{\mathfrak{i}}
\renewcommand{\jmath}{\mathfrak{j}}

\newcommand{\iso}{\cong}

\newcommand{\setcl}[2]{\left\{ \left. #1 \;\right| \; #2 \right\}}

\newcommand{\directSum}{\oplus}

\newcommand{\tensor}{\otimes}

\newcommand{\Hom}[3]{\operatorname{Hom_{#1}}\left(#2,#3\right)}
\newcommand{\End}[1]{\operatorname{End}\left(#1\right)}

\newcommand{\JW}[1]{f^{(#1)}}
\newcommand{\tr}[1]{\text{tr}\left(#1\right)}
\newcommand{\ptr}[1]{\text{tr}_1\left(#1\right)}

\ifpdf
\usepackage[pdftex]{graphicx}
\else
\usepackage[dvips]{graphicx}
\fi

\usepackage{tikz}
\usetikzlibrary{fit,calc,scopes}
\usetikzlibrary{shapes}
\usetikzlibrary{backgrounds}
\usetikzlibrary{decorations,decorations.pathreplacing}
\tikzstyle{shaded}=[fill=red!10!blue!20!gray!50!white]
\tikzstyle{shaded line}=[double=red!10!blue!20!gray!50!white, double distance=2mm, draw=black]
\tikzstyle{unshaded}=[fill=white]
\tikzstyle{unshaded line}=[double=white, double distance=2mm, draw=black]
\tikzstyle{Tbox}=[circle, draw, thick, fill=white, opaque,]
\tikzstyle{empty box}=[circle, draw, thick, fill=white, opaque, inner sep=2mm]
\tikzstyle{background rectangle}= [fill=red!10!blue!20!gray!40!white,rounded corners=2mm] 
\tikzstyle{on}=[very thick, red!50!blue!50!black]
\tikzstyle{off}=[gray]

\tikzstyle{traces}=[scale=.2, inner sep=1mm,baseline]
\tikzstyle{quadratic}=[scale=.23, inner sep=.5mm, baseline]
\tikzstyle{annular}=[scale=.6, inner sep=1mm, baseline]
\tikzstyle{make triple edge size}= [scale=.4, inner sep=1mm,baseline] 
\tikzstyle{icosahedron network}=[scale=.35, inner sep=1mm, baseline]
\tikzstyle{ATLsix}=[scale=.25, baseline]
\tikzstyle{TL12}=[scale=.15,baseline=-0.5ex]
\tikzstyle{TLEG}=[scale=.5,baseline]
\tikzstyle{PAdefn}=[scale=.7,baseline]
\tikzstyle{STrain}=[baseline=0,scale=2]

\usepackage{microtype}
\title{Non-cyclotomic fusion categories}

\author{Kim~Morrison}
\address{
}%
\email{scott@tqft.net}

\author{Noah~Snyder}
\address{
}%
\email{nsnyder@math.columbia.edu}

\address{%
\rm URLs:\stdspace \tt \url{http://tqft.net/}  \rm and \tt \url{http://math.columbia.edu/~nsnyder}
}

\date{
  First edition: the mysterious future
  This edition: \today.
}


\tikzstyle{STrain}=[baseline=0,scale=2]
\newcommand{\drawS}[3]{%
	\filldraw[fill=white,thick] (#1,#2) ellipse (3mm and 3mm);
	\node at (#1,#2) {\Large $S$};
	\path(#1,#2) ++(#3:0.37) node {$\star$};
}

\newcommand{\RainbowOne}{
	\fill[shaded] (-0.8,0) -- (-0.8,0.6) arc (180:0:0.8) -- (0.8,0) -- (0.2,0) -- (0.2,1) -- (-0.2,1) -- (-0.2,0);
	\draw (-0.8,0) -- (-0.8,0.6) arc (180:0:0.8) -- (0.8,0);
	\draw (-0.2,0) -- (-0.2,1);
	\draw (0.2,0) -- (0.2,1);
	\node at (0.45,0.5) {\footnotesize$2n-1$};
	\drawS{0}{1}{-90}
}

\newcommand{\RainbowTwo}{
	\draw (0,0) -- (0,1);
	\node at (0.15,0.5) {\footnotesize$2n$};
	\drawS{0}{1}{180}
	\filldraw[shaded] (-1,0) -- (-1,0.5) arc (180:0:1) -- (1,0)  -- (0.9,0) -- (0.9,0.5) arc (0:180:0.9) -- (-0.9,0);
}

\newcommand{\JWPlusTwo}{%
	\filldraw[fill=white,thick] (-1,-0.2) rectangle (1,0.2);
	\node at (0,0) {\Large$\JW{2n+2}$};
}

\newcommand{\JWPlusFour}{%
	\filldraw[fill=white,thick] (-1.2,-0.2) rectangle (1.2,0.2);
	\node at (0,0) {\Large$\JW{2n+4}$};
}

\newcommand{\STrainOneOne}{%
	\foreach \x in {-1,0} {
		\node[anchor=south] at (\x+0.5,1) {\footnotesize$n-1$};
		\draw (\x,1) -- (\x + 1, 1);
	}
	\foreach \x in {-1,0,1} {
		\drawS{\x}{1}{90}
	}
}

\newcommand{\STrainStrings}[2]{%
	\fill[shaded] (-0.5,0) rectangle (0.5,1);
	\draw (-0.5,1) -- (-0.5,0);
	\node[anchor=west] at (-0.5,0.5) {\footnotesize#1};
	\draw (0.5,1) -- (0.5,0);
	\node[anchor=west] at (0.5,0.5) {\footnotesize#2};
}

\newcommand{\STrainOne}{%
	\node[anchor=south] at (0,1) {\footnotesize$n-1$};
	\draw (-0.5,1) -- (0.5, 1);
	\foreach \x in {-0.5,0.5} {
		\drawS{\x}{1}{90}
	}
}

\newcommand{\STrainThreeStrings}[3]{%
	\fill[shaded] (-1,1) rectangle (1,0);
	\foreach \x in {-1,0,1} {
		\draw (\x,0) -- (\x,1);
	}
        \node [anchor=west] at (-1,0.5) {\footnotesize#1};
        \node [anchor=west] at (0,0.5) {\footnotesize#2};
	\node [anchor=west] at (1,0.5) {\footnotesize#3};
}
\begin{document}

\begin{abstract}
Etingof, Nikshych and Ostrik ask in \cite[\S 2]{MR2183279}  if every fusion category can be completely defined over a cyclotomic field.  We show that this is not the case: in particular one of the fusion categories coming from the Haagerup subfactor \cite{MR1686551} and one coming from the newly constructed extended Haagerup subfactor \cite{0909.4099} can not be completely defined over a cyclotomic field.    On the other hand, we show that the Drinfel'd center of the even part of the Haagerup subfactor is completely defined over a cyclotomic field.  We identify the minimal field of definition for each of these fusion categories, compute the Galois groups, and identify their Galois conjugates.
\end{abstract}

\maketitle

% remove table of contents for submitted version
% \setcounter{tocdepth}{1}
% \tableofcontents

%\section{Introduction}
%\input{text/intro.tex}

\section{Introduction}
In \cite[\S 2]{MR2183279}, Etingof, Nikshych and Ostrik ask if every fusion category over the complex numbers can be defined over a cyclotomic field.   More precisely, does every fusion category over $\Complex$ have a complete rational form over a cyclotomic field?  (See Section \ref{sec:background} for definitions and examples of the key notions ``rational form" and ``complete rational form.")

Their question is motivated by the following results.
\begin{itemize}
\item The representation category of any finite group has a complete rational form over a cyclotomic field. (This is a classical result of Brauer's, see \cite[\S 12.3]{MR0450380}.)
\item The semisimplified representation category of any quantum group at a root of unity has a complete rational form over a cyclotomic field.  (This follows from the usual construction of Weyl modules.)
\item The Frobenius-Perron dimension of any object in a fusion category is a cyclotomic integer \cite{MR2183279}.
\item The global dimension of a fusion category is a cyclotomic integer \cite{MR2183279}.
\item The entries of the $S$-matrix of a modular category lie in a cyclotomic field \cite{MR1266785, MR1120140}.
\end{itemize}

We answer this question in the negative.

\begin{thm} \label{thm:noncyclotomic}
The principal even part of the Haagerup subfactor and the principal even part of the extended Haagerup subfactor are fusion categories which do not have a complete rational form over any cyclotomic field.
\end{thm}

We hope that this result will eventually allow a more robust technique for establishing the ``exotic" nature of these fusion categories.  No construction that preserves cyclotomicity can produce these fusion categories starting from groups or quantum groups.  Alternately this result might suggest new techniques for constructing these fusion categories.

Let $\cH_0$ and $\cH_1$ be the Haagerup \cite{MR1686551, 0902.1294} and extended Haagerup \cite{0909.4099} subfactors which are the unique subfactors with the following principal graph pairs.
\begin{align*}
\Gamma(\cH_0) & = \left \{ \mathfig{0.25}{haagerup}, \mathfig{0.25}{dual-haagerup} \right \} \\
\Gamma(\cH_1) & = \left \{ \mathfig{0.4}{EH}, \mathfig{0.4}{dual-EH} \right \}
\end{align*}

Given a subfactor $A \subset B$ there are two tensor categories $\cS^p$ and $\cS^d$ (consisting of certain $A$-$A$ bimodules and certain $B$-$B$ bimodules respectively) called the principal even part and dual even part.  If the subfactor is finite depth then $\cS^p$ and $\cS^d$ are fusion categories over $\Complex$.  We will be looking at the fusion categories $\cH_\ell^p$ and $\cH_\ell^d$ for $\ell \in \{0,1\}$.

Denote by $D_0 = \frac{5+\sqrt{13}}{2} \simeq 4.30278$ and $D_1 = \frac{8}{3}+\frac{2}{3} \operatorname{Re} \sqrt[3]{\frac{13}{2} \left(-5-3 i \sqrt{3}\right)} \simeq 4.3772$ the Jones indices of $\cH_0$ and $\cH_1$.  Fix 
\begin{align*}
\lambda_0 & = i \sqrt{\frac{-1+\sqrt{13}}{6}} \simeq 0.658983 i \\
\intertext{and}
\lambda_1& = \sqrt{-\frac{1}{5}+2 \operatorname{Re} \sqrt[3]{\frac{117-  65 i \sqrt{3}}{2250}}} \simeq 0.648585 i.
\end{align*}
%, where we take the cube root in the upper half plane
 Let $\zeta_m$ denote the primitive $m$th root of unity $\exp(2\pi i/m)$.  Note that
\begin{align*}
D_0 & = 2 -\zeta_{13}^{2} -\zeta_{13}^5 -\zeta_{13}^6 -\zeta_{13}^7-\zeta_{13}^8-\zeta_{13}^{11}\\
D_1 & = 3+ \zeta_{13}^2  + \zeta_{13}^3 + \zeta_{13}^{10} +  \zeta_{13}^{11}
\end{align*}
so $D_\ell \in \Rational(\zeta_{13})$ while $\lambda_\ell$ is not cyclotomic.  In fact, $\Rational(\lambda_\ell)$ is not Galois.   The Galois group of the Galois closure of $\Rational(\lambda_0)$ is the dihedral group of order $8$, and the Galois group of the Galois closure of $\Rational(\lambda_1)$ is $\Integer/2\Integer \wr \Integer/3\Integer \cong \Integer/2\Integer \times A_4$.

\begin{thm} \label{thm:detailed}
The following statements hold for $\ell = 0,1$.
\begin{enumerate}
\item \label{pa:construct} The even parts $\cH_\ell^p$ and $\cH_\ell^d$ each have a (possibly incomplete) rational form over $\Rational(D_\ell)$.
%\noah{It is very hard to make this sentence grammatical.  The old version made it sound like you could only define *both* of them when you had $D_\ell$ but maybe you could define one or the other.}
\item \label{pa:dualprojs} The dual even part $\cH_\ell^d$ has a complete rational form over $k$ if and only if $D_\ell \in k$.
\item \label{pa:principalprojs} The principal even part $\cH_\ell^p$ has a complete rational form over $k$ if and only if $\lambda_\ell \in k$.
\item \label{pa:center} The Drinfel'd center $Z(\cH_0^p) \cong Z(\cH_0^d)$ has a complete rational form over $k$ as a ribbon fusion category if and only if $\zeta_{39} \in k$.
\end{enumerate}
\end{thm}

Theorem \ref{thm:noncyclotomic} follows immediately from part \ref{pa:principalprojs} of Theorem \ref{thm:detailed}.  Part 4 is of interest because it means these results do not exclude the possibility that every braided fusion category is defined over a cyclotomic field.  We prove part \ref{pa:construct} and the ``if" direction of parts \ref{pa:dualprojs} and \ref{pa:principalprojs} in Section \ref{sec:construction}.  We prove the ``only if" direction of parts \ref{pa:dualprojs} and \ref{pa:principalprojs} in Section \ref{sec:noncyclotomic}.  We prove part \ref{pa:center} in Section \ref{sec:center}.

The main technique in this paper is to show that, in the context of fusion categories associated to 3-supertransitive subfactors, the (correctly normalized) ``twisted moments" of any $\Rational$-linear combination of projections gives an element of the base field of any complete rational form.  These twisted moments can be computed using techniques from Jones's preprint \cite{quadratic}.  In the construction of the Haagerup subfactor by Peters \cite{0902.1294}, the moments and twisted moments of the ``generator" are the only scalars needed to define the subfactor.  The $\lambda_\ell$ above are $\Rational(D_\ell)$ multiples of the third twisted moments of $\frac{1}{2}\left(\id_P-\id_Q\right)$ (where $P$ and $Q$ are the two simple objects immediately after the branch).

There is a third subfactor with index in the interval $(2, 3+\sqrt{3})$ called the Asaeda-Haagerup subfactor \cite{MR1686551}.  Our techniques do not give an obstruction to cyclotomicity for either of the fusion categories coming from the Asaeda-Haagerup subfactor, because the analogous moments and twisted moments are cyclotomic.  However, since there is not yet a construction of the Asaeda-Haagerup planar algebra following the Jones-Peters approach \cite{quadratic, 0902.1294, 0909.4099}, the lack of obstruction does not guarantee that the even parts of the Asaeda-Haagerup subfactor are cyclotomic.

The authors would like to thank Emily Peters for teaching us about the Haagerup planar algebra and Pavel Etingof for encouraging us to write this paper.  We'd like to thank Victor Ostrik, Ben Webster, and Pasquale Zito for suggesting arguments which we used to improve Section \ref{sec:center}.  (For Pasquale Zito's suggestions, see Math Overflow \url{http://mathoverflow.net/questions/17641/}.)  In addition we would also like to thank Stephen Bigelow, Vaughan Jones, Dmitri Nikshych, and Dylan Thurston for helpful conversations.  Kim Morrison was at the Miller Institute for Basic Research at UC Berkeley during this work, and Noah Snyder was supported by an NSF Postdoctoral Fellowship at Columbia University.

\section{Background}
\label{sec:background}
%!TEX root = ../article.tex

\subsection{Fusion categories and fields of definition}
Let $k$ be a field.  An object in an additive category is called simple if it has no non-trivial proper subobjects.  An additive category is called semisimple if every object is a direct sum of simple objects (and in particular, every indecomposable object is simple).  A category is called idempotent complete (or Karoubian or psuedoabelian) if every idempotent has an image (that is, a subobject which the idempotent factors through).  It is easy to see that any idempotent complete additive semisimple category is abelian.  A \emph{split semisimple category over $k$} is a semisimple category over $k$ such that every simple object $X$ is split simple, that is $\End{X} = k$.  If $k$ is an algebraically closed field, then any semisimple category over $k$ is automatically split.

A \emph{fusion category over $k$} is a $k$-linear abelian semisimple rigid monoidal category with finitely many isomorphism classes of simple objects.  A \emph{split fusion category over $k$} is a fusion category over $k$ which is split semisimple.  (Warning, some authors require that all fusion categories be split.)

\begin{example}
Consider $\Real[\Integer/3 \Integer]\text{-mod}$.   This is a fusion category  over $\Real$ with two objects: the trivial module and the $2$-dimensional representation (where the generator acts by $120$-degree rotation).  It is not split fusion because the endomorphism algebra of the $2$-dimensional representation is $\Complex$.
\end{example}

Suppose that $k<K$ is an inclusion of fields and that $\cC_K$ is a semisimple abelian category over $K$ possibly with some fixed additional structures (e.g. it's a monoidal rigid category or a braided category).  A \emph{rational form} of $\cC$ over $k$ is a semisimple abelian category $\cC_k$ over $k$ (again with the same structure) together with an equivalence between the idempotent completion of $\cC_k \otimes_k K$ is and $\cC_K$ (this equivalence should preserve the additional structure).  A \emph{complete rational form} of $\cC$ over $k$ is a rational form $\cC_k$ such that $\cC_k \otimes_k K \cong \cC_K$, in other words, a complete rational form is a rational form such that $\cC_k \otimes_k K$ is already idempotent complete.

%Suppose that $k<K$ is an inclusion of fields and suppose that $\cC_K$ is a semisimple abelian tensor category over $K$.  A \emph{rational form} of $\cC$ over $k$ is a fusion category $\cC_k$ over $k$ such that the idempotent additive completion of $\cC_k \otimes_k K$ is equivalent to $\cC_K$.  A \emph{complete rational form} of $\cC$ over $k$ is a fusion category $\cC_k$ such that $\cC_k \otimes_k K \cong \cC_K$, in other words, a complete rational form is a rational form such that $\cC_k \otimes_k K$ is already idempotent complete.

There are several reasons for considering incomplete rational forms.  First, many constructions do not preserve completeness.  For example, the category of $G$-graded rational vector spaces is a complete rational form of the category of $G$-graded complex vector spaces, while the center of the former category is only an incomplete rational form of the center of the latter.  Second, any fusion category is the category of representations of a weak Hopf algebra $A$ \cite{MR1976459, 0206113, 9904073, MR2522429} and if that weak Hopf algebra has an rational form $A_k$ then $A_k-\text{mod}$ is, in general, an incomplete rational form for $A-\text{mod}$.  Third, the notion of incomplete rational form arises naturally in the context of planar algebras.  Finally, it can be convenient to show that $\cC_k$ is a complete rational form by first showing that it is a rational form (which is often easy for the above reasons) and then explicitly checking completeness.

Notice that if $\cC_K$ is a split fusion category over $K$ and $\cC_k$ is a rational form over $k < K$ then $\cC_k$ is complete if and only if $\cC_k$ is a split fusion category.

%A \emph{split rational form} is a rational form $\cC_k$ which is a split.  We say that $\cC_K$ is \emph{defined over} $k$ if there exists a split rational form $\cC_k$.  We will be interested in the case $K = \Complex$ and $k$ a number field.  An important equivalent condition to $\cC_k$ being split is that base extension give an isomorphism on Grothendieck groups $K_0(\cC_k) \cong K_0(\cC_K)$.  In particular, any rational form for $\cC_K$ categorifies the same fusion ring.

If $\cC_K$ is a split fusion category then there is a more concrete description of having a complete rational form (pointed out to us by Victor Ostrik).  A split fusion category $\cC_K$ can be completely described by a collection of vector spaces $\Hom{}{V_a \otimes V_b}{V_c}$ over $K$, and associativity maps between the appropriate tensor products of these spaces.   The fusion category $\cC_K$ has a complete rational form over $k$ if and only if there exists a basis for each of these vector spaces such that all the associativity maps are given by matrices with entries in $k$.

\begin{example} \label{ex:graded}
$\cC_\Complex = \Complex[\Integer/3 \Integer]\text{-mod}$ has a complete rational form over $\Rational$.  To see this notice that $\Integer/3\Integer$ is abelian, and hence $\Complex[\Integer/3 \Integer]\text{-mod}$ is equivalent to the category of $\Integer/3 \Integer$-graded complex vector spaces.  We can take $\cC_\Rational$ to be the category of $\Integer/3 \Integer$-graded rational vector spaces.
\end{example}

\begin{example} \label{ex:incomplete}
Consider $G$ a finite group.  Notice that all of the representations of a group $G$ are defined over a field $k$ if and only if $k[G]\text{-mod}$ is split.  All representations of any finite group $G$ are defined over $\Rational(\zeta_{n})$, where $n$ is the exponent of $G$.  Therefore, $\Complex[G]\text{-mod}$ has a complete rational form over a cyclotomic field.  On the other hand as Example \ref{ex:graded} shows, it is possible for $\Complex[G]\text{-mod}$ to be defined over a smaller field than the minimal field of definition for all its representations.
\end{example}

If $\cC$ is a $k$-linear category, and $k < K$ is an inclusion of fields and the idempotent additive completion of $\cC \tensor_k K$ is semisimple, then $\cC$ is semisimple.  This follows from the characterization of semisimplicity in terms of semisimplicity of the endomorphism algebras, given in the next Lemma, and the equivalence of semisimplicity of artinian rings with the absence of nilpotent ideals \cite[p. 203]{MR1009787}.  The same is not true for split semisimplicity as  $\Real[\Integer/3 \Integer]\text{-mod}$ is not split semisimple while  $\Complex[\Integer/3 \Integer]\text{-mod}$ is.  The converse is also false at least when $k<K$ is inseparable \cite[Question 5.1]{MR1995781}.

\begin{lem}
A $k$-linear additive idempotent complete category $\cC$ is semisimple if and only if the endomorphism algebra of every object is a semisimple algebra.  The simple objects are exactly those objects whose endomorphism algebras are division rings.
\end{lem}
\begin{proof}
By Schur's lemma, if $\cC$ is semisimple then the endomorphism algebra of an object is a sum of matrix algebras over division rings, which is a semisimple algebra.

In the other direction, suppose that $\cC$ has all endomorphism algebras semisimple.  Since the category is idempotent complete, one may use the projections in the endomorphism algebras to decompose any object into a direct sum of objects whose endomorphism algebras are division rings.  If $X$ and $Y$ are non-isomorphic objects whose endomorphism algebras are division rings, the semisimplicity of $\End{X\oplus Y}$ shows that there are no nonzero maps between $X$ and $Y$.  Hence all objects whose endomorphism algebras are division rings are simple and every object is semisimple.
\end{proof}

%If $k<\Complex$ and $\cC$ is a fusion category over $\Complex$ then a \emph{rational form} of $\cC$ over $k$ is a fusion category $\cC_k$ over $k$ such that $\cC_k \otimes_k \Complex \cong \cC$ and such that the Grothendieck group $K_0(\cC_k) = K_0(\cC)$ (or equivalently such that every simple object in $\cC_k$ is absolutely simple). The latter condition is necessary in order for this concept to match up with field of definition of representations of algebras, for example even though the representation theory of $\Integer/3\Integer$ is not defined over $\Real$ nonetheless $\Real[\Integer/3\Integer]\text{-mod} \otimes_\Real \Complex \cong \Complex[\Integer/3\Integer]\text{-mod}$. \scott{Clunky sentence} In particular with our terminology $\Complex[\Integer/3\Integer]\text{-mod}$ does not have a rational form over $\Real$.  If $\cC$ has a rational form over $k$ we say that $\cC$ is \emph{defined over} $k$.

\subsection{Subfactors and category theory}

The main examples in this paper come from the theory of subfactors.  A subfactor is an inclusion $A \subset B$ of von Neumann algebras with trivial centers.  The applications of subfactor theory to tensor categories involve subfactors of finite index.

From such a subfactor $A \subset B$, we can construct a pair of $\Complex$-linear abelian semisimple rigid tensor categories (that is, categories that only fail to be fusion categories by possibly having infinitely many non-isomorphic simple objects).  See \cite{MR1424954} for more details. These tensor categories are  the subcategory of the category of $A\text{-}A$ bimodules tensor generated by ${}_A B_A$ and the subcategory of $B\text{-}B$ bimodules tensor generated by ${}_B B_A \otimes_A {}_A B_B$.  These tensor categories are fusion if and only if the subfactor is finite depth.

It will not be important to this paper to understand subfactors or their bimodules. Although our main examples come from subfactors, these subfactors can be constructed via their planar algebras, and we can define the two associated tensor categories directly from the planar algebra.

\subsection{Shaded planar algebras}

One of the main tools for understanding monoidal categories (and more generally $2$-categories) is the diagram calculus (pioneered by Penrose \cite{MR0281657}, Joyal-Street \cite{MR1113284}, Reshetikhin-Turaev \cite{MR1036112}, etc.) of string diagrams.  Planar algebras are a version of this diagram calculus which allows you to study monoidal categories from the point of view of a particular object, where you only consider strings labelled by that object.

In particular, given a subfactor the resulting string diagrams have regions which are checkerboard shaded (corresponding to $A$ and $B$) and unoriented unlabelled strings (depending on the shading this string represents either ${}_A B_B$ or ${}_B B_A$).  The structure here is a called a ``shaded planar algebra."

We sketch the definition here.  For further details, see  \cite[\S 2]{MR1865703}, \cite[\S 0]{math.QA/9909027}, or \cite{MR1957084}.
%(The original definition is phrased in terms of colored operads, but essentially the same).

\begin{defn}
A {\em (shaded) planar tangle} 
has an outer disk,
a finite number of inner disks,
and a finite number of non-intersecting strings.
A string can be either a closed loop
or an edge with endpoints on boundary circles.
We require that
there be an even number of endpoints on each boundary circle,
and a checkerboard shading of
the regions in the complement of the interior disks.
We further require that there be a marked point on the boundary of each disk, and that the inner disks are ordered.

Two planar tangles are considered equal if they are isotopic (not necessarily rel boundary).

Here is an example of a planar tangle.
$$%
%\beginpgfgraphicnamed{\pathtotrunk diagrams/tikz/#1-external}%
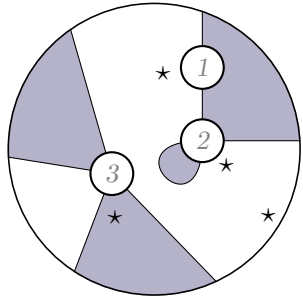
\begin{tikzpicture}[scale=.65]
	\clip (0,0) circle (3cm);
	
	\begin{scope}[shift=(10:1cm)]	
		\draw[shaded] (0,0)--(0:6cm)--(90:6cm)--(0,0);	
		\draw[shaded] (0,0) .. controls ++(180:2cm) and ++(-90:2cm) .. (0,0);
	\end{scope}
	
	\draw[shaded] (-150:1cm) -- (120:4cm) -- (180:4cm) -- (-150:1cm);
	\draw[shaded] (-150:1cm) -- (-120:4cm) -- (-60:4cm) -- (-150:1cm);
	
	\begin{scope}[shift=(10:1cm)]	
		\node at (0,0) [Tbox, inner sep=1mm] {\small{\textcolor{gray}{2}}};
		\node at (90:1.5cm) [Tbox, inner sep=1mm] {\small{\textcolor{gray}{1}}};
		\node at (-45:.7cm) {$\star$};
		\node at (120:1.6cm) {$\star$};
	\end{scope}
	\node at (-150:1cm) [Tbox, inner sep=1 mm] {\small{\textcolor{gray}{3}}};
	\node at (-120:1.6cm) {$\star$};
	\node at (-30:2.7cm) {$\star$};
	
	\draw[very thick] (0,0) circle (3cm);
\end{tikzpicture}%
%\endpgfgraphicnamed
$$

Planar tangles can be composed by placing one planar tangle inside an interior disk of another,
lining up the marked points,
and connecting endpoints of strands.
The numbers of endpoints and the shadings must match up appropriately.  This composition turns the collection of planar tangles into a colored operad.
\end{defn}

\begin{defn}
A {\em (shaded) planar algebra} over a field $k$ consists of
\begin{itemize}
\item A family of vector spaces $\{V_{(n,\pm,)}\}_{n\in \Natural}$ over $k$, called the positive and negative $n$-box spaces.
\item For each planar tangle, a multilinear map $V_{n_1, \pm_1} \otimes \ldots \otimes V_{n_j, \pm_j} \rightarrow V_{n_0, \pm_0}$ where $n_i$ is half the number of endpoints on the $i$th interior boundary circle, $n_0$ is half the number of endpoints on the outer boundary circle, and the signs $\pm$ are positive (respectively negative) when the marked point on the corresponding boundary circle is in an unshaded (respectively shaded) region.
\end{itemize}

For example, the planar tangle above gives a map $$V_{1,+} \otimes V_{2,+} \otimes V_{2,-} \rightarrow V_{3,+}.$$

The linear map associated to a `radial' tangle (with one inner disc, radial strings, and matching marked points) must be the identity.
We require that the action of planar tangles be compatible with composition of planar tangles.  In other words, composition of planar tangles must correspond to the obvious composition of multilinear maps.  \end{defn}

We will refer to an element of $V_{n,\pm}$ (and specifically $V_{n,+}$, unless otherwise stated) as an ``$n$-box.''

We make frequent use of three families of planar tangles
called multiplication, trace, and tensor product, which are shown in Figure \ref{fig:timestracetensor}.
``Multiplication'' gives an associative product
$V_{n,\pm} \otimes V_{n,\pm} \rightarrow V_{n,\pm}$.
``Trace'' gives a map $V_{n,\pm} \rightarrow V_{0,\pm}$.
``Tensor product'' gives an associative product
$V_{m,\pm} \otimes V_{n,\pm} \rightarrow V_{m+n,\pm}$ if $m$ is even,
or $V_{m,\pm} \otimes V_{n,\mp} \rightarrow V_{m+n,\pm}$ if $m$ is odd.

\begin{figure}[ht]
$$%
%\beginpgfgraphicnamed{\pathtotrunk diagrams/tikz/#1-external}%
 \begin{tikzpicture}[PAdefn]
	\clip [draw] (2,2) arc (0:180:2cm) -- (-2,-2) arc (-180:0:2cm) -- (2,2);
	
	%first draw the lines
	\draw (0,2) .. controls ++(-150:1.5cm) and ++(150:1.5cm) .. (0,-2) .. controls ++(110:1.5cm) and ++(-110:1.5cm) .. (0,2);

	\draw (0,2) .. controls ++(-30:1.5cm) and ++(30:1.5cm) .. (0,-2);
	
	\draw (0,2) -- +(110:3cm) -- +(130:3cm) -- (0,2);
	\draw (0,2) -- ++(50:3cm);
		
	\draw (0,-2) -- ++(-50:3cm);
	\draw (0,-2) -- +(-110:3cm) -- +(-130:3cm) -- (0,-2);
	
	%decorate with dots, labels and stars
	\node at (.4,0) {$\ldots$};
	\node at (.2,3.2) {$\dots$};
	\node at (.2,-3.2) {$\dots$};

	\node at (0,2) [Tbox,inner sep=1.4mm] (A) {\small{\textcolor{gray}{1}}};
	\node at (0,-2) [Tbox,inner sep=1.4mm] (B) {\small{\textcolor{gray}{2}}};
	\node at (A.180) [left] {$\star$};
	\node at (B.180) [left] {$\star$};
	\node at (-1.5,2.5)  {$\star$};	
	
	%redraw boundary
	\draw[ultra thick] (2,2) arc (0:180:2cm) -- (-2,-2) arc (-180:0:2cm) -- (2,2);
\end{tikzpicture}%
%\endpgfgraphicnamed
 \quad , \quad %
%\beginpgfgraphicnamed{\pathtotrunk diagrams/tikz/#1-external}%
\begin{tikzpicture}[scale=.25,baseline]
	\clip (8,6) arc (0:180:6cm) -- (-4,-6) arc (-180:0:6cm) -- (8,6);

%	\filldraw[shaded] (0,0) .. controls ++(-157:3cm) and ++(157:3cm) .. (0,0) .. controls ++(112:2cm) and ++(-112:2cm) .. (0,0);
%	\filldraw[shaded] (0,0) .. controls ++(-22:3cm) and ++(22:3cm) .. (0,0) .. controls ++(67:2cm) and ++(-67:2cm) .. (0,0);
%	\draw (0,0) .. controls ++(22:3cm) and ++(90:4cm) .. (3,0) .. controls ++(-90:4cm) and ++(-22:3cm) .. 
	\draw (0,0) .. controls ++(-67:6cm) and ++(-90:5cm) .. (3,0) .. controls ++(90:5cm) and ++(67:6cm) .. (0,0);
	\draw (0,0) .. controls ++(130:4cm) and ++(90:15cm) .. (5,0) .. controls ++(-90:15cm) and ++(-130:4cm) .. (0,0) .. controls ++(-157:7cm) and ++(-90:20cm) .. (6,0) .. controls ++(90:20cm) and ++(157:7cm) .. (0,0);

	\node at (0,0) [Tbox,inner sep=2mm] (T1) {};
	\node at (T1.180) [left] {$\star$};
	\node at (4.1,0) {$\cdots$};
	
	\draw[ultra thick] (8,6) arc (0:180:6cm) -- (-4,-6) arc (-180:0:6cm) -- (8,6);
\end{tikzpicture}%
%\endpgfgraphicnamed
 \quad , \quad %
%\beginpgfgraphicnamed{\pathtotrunk diagrams/tikz/#1-external}%
 \begin{tikzpicture}[PAdefn]
	\clip [draw] (2,2) arc (90:-90:2cm) -- (-2,-2) arc (-90:-270:2cm) -- (2,2);
	
	%first draw the lines	
	\draw (2,0) -- +(110:3cm) -- +(130:3cm) -- (2,0);
	\draw (2,0) -- ++(50:3cm);
	\draw (2,0) -- ++(-50:3cm);
	\draw (2,0) -- +(-110:3cm) -- +(-130:3cm) -- (2,0);
		
	\draw (-2,0) -- +(110:3cm) -- +(130:3cm) -- (-2,0);
	\draw (-2,0) -- ++(50:3cm);
	\draw (-2,0) -- ++(-50:3cm);
	\draw (-2,0) -- +(-110:3cm) -- +(-130:3cm) -- (-2,0);
	
	%decorate with dots, labels and stars
	\node at (-1.8,1.2) {$\dots$};
	\node at (-1.8,-1.2) {$\dots$};
	\node at (2.2,1.2) {$\dots$};
	\node at (2.2,-1.2) {$\dots$};

	\node at (2,0) [Tbox,inner sep=1.4mm] (A) {\small{\textcolor{gray}{2}}};
	\node at (-2,0) [Tbox,inner sep=1.4mm] (B) {\small{\textcolor{gray}{1}}};
	\node at (A.180) [left] {$\star$};
	\node at (B.180) [left] {$\star$};
	\node at (-3.7,0)  {$\star$};	
	
	%redraw boundary
	\draw[ultra thick] (2,2) arc (90:-90:2cm) -- (-2,-2) arc (-90:-270:2cm) -- (2,2);
\end{tikzpicture}%
%\endpgfgraphicnamed
$$
\caption
  {The multiplication, trace, and tensor product tangles.}
\label{fig:timestracetensor}
\end{figure}
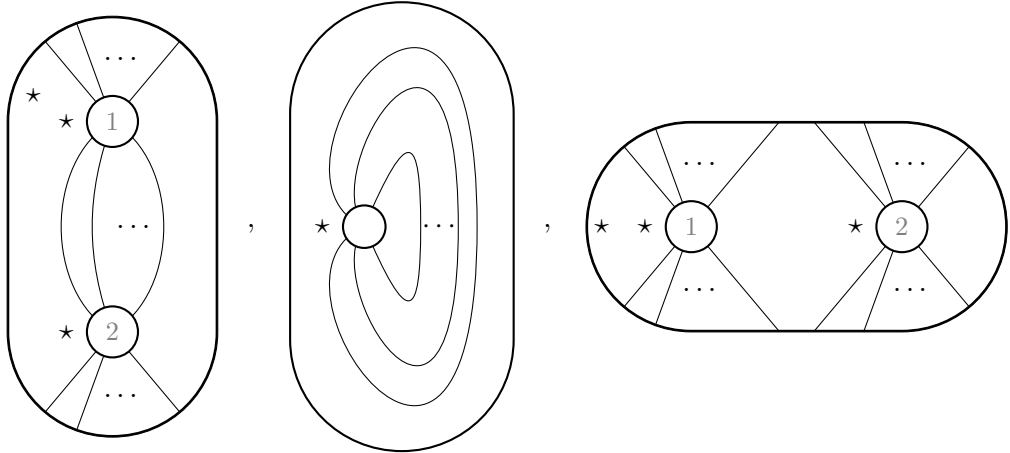

The (shaded or unshaded) empty diagrams can be thought of as elements of $V_{0,\pm}$, since the `empty tangle' induces a map from the empty tensor product $\Complex $ to the space $V_{0,\pm}$.  If the space $V_{0,\pm}$ is one dimensional
then we can identify it with $\Complex$  by sending the empty diagram to one.

%\begin{defn}
%A planar algebra is called simple if $\dim V_{0,\pm} = 1$.
%\end{defn}

\begin{defn}
A planar algebra is called \emph{irreducible} if $\dim V_{0,\pm} = 1$ and $\dim V_{1,\pm} = 1$.
\end{defn}

\begin{defn}\label{defn:uncappable}
An $n$-box $S$ is {\em uncappable}  if $\epsilon_i(S)=0$ for all $i=1,\cdots, 2n$ where
 $$\epsilon_1 = %
%\beginpgfgraphicnamed{\pathtotrunk diagrams/tikz/#1-external}%
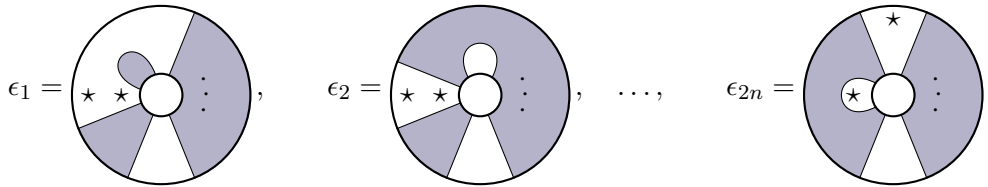
\begin{tikzpicture}[annular]
	\clip (0,0) circle (2cm);

	\filldraw[shaded] (0,0) .. controls ++(170:2cm) and ++(100:2cm) .. (0,0);
	\filldraw[shaded] (-158:4cm)--(0,0)--(-112:4cm);

	\draw[shaded] (68:4cm)--(0,0)--(-68:4cm)--(0:10cm);
	
	\draw[ultra thick] (0,0) circle (2cm);
	
	\node at (0,0)  [empty box] (T) {};
	\node at (T.180) [left] {$\star$};
	\node at (180:2cm) [right] {$\star$};
	\node at (0:1cm) {$\cdot$};
	\node at (20:1cm) {$\cdot$};
	\node at (-20:1cm) {$\cdot$};
	
\end{tikzpicture}%
%\endpgfgraphicnamed
 \, , \qquad 
 \epsilon_2 = %
%\beginpgfgraphicnamed{\pathtotrunk diagrams/tikz/#1-external}%
\begin{tikzpicture}[annular]
	\clip (0,0) circle (2cm);

	\filldraw[shaded] (158:4cm) -- (0,0) .. controls ++(130:2cm) and ++(50:2cm) .. (0,0)--(-68:4cm) arc (-68:158:4cm);
	\filldraw[shaded] (-158:4cm)--(0,0)--(-112:4cm);

	\draw[ultra thick] (0,0) circle (2cm);
	
	\node at (0,0)  [empty box] (T) {};
	\node at (T.180) [left] {$\star$};
	\node at (180:2cm) [right] {$\star$};
	\node at (0:1cm) {$\cdot$};
	\node at (20:1cm) {$\cdot$};
	\node at (-20:1cm) {$\cdot$};
	
\end{tikzpicture}%
%\endpgfgraphicnamed
 \, , \quad \ldots , \qquad 
 \epsilon_{2n}= %
%\beginpgfgraphicnamed{\pathtotrunk diagrams/tikz/#1-external}%
\begin{tikzpicture}[annular]
	\clip (0,0) circle (2cm);

	\filldraw[shaded] (112:4cm) -- (0,0) .. controls ++(140:2cm) and ++(-140:2cm) .. (0,0)--(-112:4cm) arc (-112:-202:4cm);
	
	\draw[shaded] (68:4cm)--(0,0)--(-68:4cm)--(0:10cm);
	
	\draw[ultra thick] (0,0) circle (2cm);
	
	\node at (0,0)  [empty box] (T) {};
	\node at (T.180) [left] {$\star$};
	\node at (90:2cm) [below] {$\star$};
	\node at (0:1cm) {$\cdot$};
	\node at (20:1cm) {$\cdot$};
	\node at (-20:1cm) {$\cdot$};
	
\end{tikzpicture}%
%\endpgfgraphicnamed
 \, . $$
 
 We say $S$ is a {\em rotational eigenvector with eigenvalue $\omega$} if   $\rho(S)= \omega S$ where
  $$\rho=%
%\beginpgfgraphicnamed{\pathtotrunk diagrams/tikz/#1-external}%
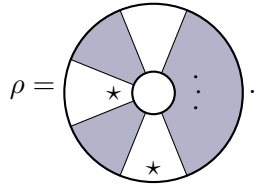
\begin{tikzpicture}[annular]
	\clip (0,0) circle (2cm);

	\filldraw[shaded] (158:4cm)--(0,0)--(112:4cm);
	\filldraw[shaded] (-158:4cm)--(0,0)--(-112:4cm);
	
	\draw[shaded] (68:4cm)--(0,0)--(-68:4cm)--(0:10cm);
	
	\draw[ultra thick] (0,0) circle (2cm);
	
	\node at (0,0)  [empty box] (T) {};
	\node at (T.180) [left] {$\star$};
	\node at (-90:2cm) [above] {$\star$};
	\node at (0:1cm) {$\cdot$};
	\node at (20:1cm) {$\cdot$};
	\node at (-20:1cm) {$\cdot$};
\end{tikzpicture}%
%\endpgfgraphicnamed
 \, .$$
 Note that $\omega$ must be a $n^{th}$ root of unity.
\end{defn}

Given a shaded planar algebra we can recover two tensor categories called its even parts.  However to describe this we need to deal with an annoying technical point.  Recall that the planar algebra only captures maps between tensor products of a fixed generating object.  If every object of a semisimple tensor category appears as a summand of some tensor power of a fixed object, then the tensor category can be recovered from the full subcategory of these tensor powers by taking the idempotent completion.

\begin{defn}
The idempotent completion of a tensor category $\check{\cC}$ is a tensor category $\cC$ which contains $\check{\cC}$ as a full sub-category. The objects of $\cC$ are pairs $(o,p)$, where $o$ is an object of $\check{\cC}$ and $p:o \to o$ is an idempotent in $\check{\cC}$.  We define $$\Hom{\cC}{(o,p)}{(o',p')} = \setcl{ f \in \Hom{\check{\cC}}{o}{o'} }{ f p = f = p' f}$$ and inherit composition and tensor product.  (When we say $\check{\cC}$ is a full sub-category of $\cC$ we are implicitly identifying $o$ with $(o,\id_o)$.)
\end{defn}

\begin{defn}
If $\cP$ is a shaded planar algebra, let $\check{\cP}^p$ be the category whose objects are even integers and whose morphisms are given by
$$\Hom{\check{\cP}^p}{2m}{2n} = V_{m+n,+}.$$
 Let $\cP^p$, called the principal even part of $\cP$ be the idempotent completion of the additive completion of $\check{\cP}^p$.

Similarly define $\check{\cP}^d$ whose objects are even integers and whose morphisms are given by
$$\Hom{\check{\cP}^d}{2m}{2n} = V_{m+n,-},$$ and the dual even part $\cP^d$ to be the idempotent completion of the additive completion of $\check{\cP}^d$.
\end{defn}

A planar algebra is called \emph{unitary} if there is an antilinear adjoint operation $*$ on each $V_{n,\pm}$, compatible with the adjoint operation on planar tangles given by reflection such that the sesquilinear form $\langle x,y \rangle = \tr{x y^*}$ is positive definite.  The even parts of a unitary planar algebra are semisimple because all the endomorphism algebras are finite dimensional $\Complex^*$ algebras which are necessarily semisimple.  Often in order to construct a unitary planar algebra from a non-unitary planar algebra we quotient out by the radical of $\langle x,y \rangle = \tr{x y^*}$, which is called the ideal of negligible morphisms.

Suppose that $k<K$ is an inclusion of fields and suppose that $\cP$ is a planar algebra over $K$.  A \emph{rational form} of $\cP$ over $k$ is a planar algebra $\cP_k$ over $k$ such that $\cP_k \otimes_k K$ is isomorphic to $\cP$.  If $\cP_k$ is a rational form for $\cP$ then the corresponding fusion categories $\cP_k^p$ and $\cP_k^d$ are rational forms for $\cP^p$ and $\cP^d$.  The rational form $\cP_k^p$ is a complete rational form for $\cP^p$ if every isomorphism class of projection in $V_{m+n,+}(\cP)$ has a representative coming from $V_{m+n,+}(\cP^p)$ (and similarly for the dual even part).

\subsection{Principal graphs}

Given a unitary irreducible shaded planar algebra $\cP$ the principal graphs are a pair of bipartite graphs which together encode the fusion rules for tensoring with the single strand.

The principal graph has even vertices corresponding to isomorphism classes of simple projections in $V_{\text{even},+}$ and odd vertices corresponding to the isomorphism classes of simple projections in $V_{\text{odd},+}$.  An even vertex $V$ and an odd vertex $W$ are connected by $\dim \Hom{}{W \otimes X}{V}$ edges, where $X$ is the single strand.  The dual principal graph has even vertices corresponding to isomorphism classes of simple projections in $V_{\text{even},-}$, odd vertices corresponding $V_{\text{odd},-}$, and edges given by the fusion with the single strand of the opposite shading.

A subfactor is called \emph{finite depth} if the principal graph is finite.

\subsection{The lopsided normalization}
\label{sec:lopsided}%
In an irreducible shaded planar algebra $\cP$ over a field $k$, the shaded inside and shaded outside circles each evaluate to a scalar multiple of the empty diagram. These multiples are called the shaded and unshaded moduli. The product of the two moduli is called the index of $\cP$.

If the two moduli are equal, we say the planar algebra is spherical. (In particular, because in this section we are assuming irreducibility, this condition implies the usual notion of a planar algebra being spherical.) If the shaded modulus is equal to one, we say the planar algebra is lopsided, and the unshaded modulus is the index.

Any irreducible shaded planar algebra $\cP$ over $\Complex$ is part of a family $\left\{\cP_x\right\}_{x \in \Complex}$ of planar algebras with $\cP_1 = \cP$, and $\cP_x$ given by changing the action of the planar operad by a factor of $x^k$, where $k$ is the signed count of critical points in the strands which are shaded above and unshaded below (minima counting positively, maxima counting negatively). The index is constant across this family, but the shaded and unshaded moduli scale by $x$ and $x^{-1}$. 

It may be that the planar algebra $\cP_x$ has a rational form over a field $k$ for certain $x$, but not all $x$.  In particular, given a planar algebra $\cC$ over $k$ with unshaded modulus $x \in k$, then $\cC_{x^{-1}}$ is a lopsided planar algebra over $k(x)=k$.  On the other hand, given a lopsided planar algebra $\cC$ over $k$ with index $D \in k$ the corresponding spherical planar algebra $\cC_{\sqrt{D}}$ is in general only defined over $k(\sqrt{D})$.  The fact that the field of definition may need to increase to ensure being spherical but does not need to increase to ensure being lopsided should encourage you to prefer lopsided planar algebras over spherical planar algebras. (It's also natural from a subfactor perspective: $B$ as an $A-B$ bimodule has left-dimension the index and right-dimension $1$.)

\newcommand{\trl}[1]{\operatorname{tr}_{L}\left(#1\right)}
\newcommand{\trr}[1]{\operatorname{tr}_{R}\left(#1\right)}

Any element $v \in \cP_{n,\pm}$ has a left trace and a right trace, satisfying
\begin{align*}
\trl{v} & =
	\begin{cases}
		\trr{v} & \text{if $n$ is even} \\
		\frac{d_{\mp}}{d_{\pm}} \trr{v} & \text{if $n$ is odd}
	\end{cases}
\end{align*}
where $d_+$ is the unshaded modulus and $d_-$ is the shaded modulus. When $\cP$ is spherical, the left and right traces coincide. We will sometimes refer to the trace of an idempotent as its dimension, and if a side is not specified we always intend the right trace.

\begin{example}
The two strand Jones-Wenzl idempotents in lopsided Temperley-Lieb are 
\begin{align*}
\begin{tikzpicture}[baseline=-1ex]
\node[fill=white, draw, rectangle] (left) at (0,0) {$\JW{2}$};
\begin{pgfonlayer}{background} 
\path[fill=black!20]
	($(left.240)+(0,-0.5)$) -- ($(left.120)+(0,0.5)$) -- ($(left.60)+(0,0.5)$) -- ($(left.300)+(0,-0.5)$);
\draw ($(left.240)+(0,-0.5)$) -- ($(left.120)+(0,0.5)$);
\draw  ($(left.60)+(0,0.5)$) -- ($(left.300)+(0,-0.5)$);
\end{pgfonlayer}
\end{tikzpicture}
& = 
\begin{tikzpicture}[baseline=-1ex]
\clip (-0.75,-0.8) rectangle (0.75,0.8);
\path[fill=black!20]
	 (-0.5,-0.8) -- (-0.5,0.8) -- (0.5,0.8) -- (0.5,-0.8);
\draw (-0.5,-0.8) -- (-0.5,0.8);
\draw  (0.5,0.8) -- (0.5,-0.8);
\end{tikzpicture}
 - 
\begin{tikzpicture}[baseline=-1ex]
\clip (-0.75,-0.8) rectangle (0.75,0.8);
\path[draw, fill=black!20]
	 (-0.5,-0.8) .. controls ++(0,0.7) and ++(0,0.7) .. (0.5,-0.8);
\path[draw, fill=black!20]
	 (-0.5,0.8) .. controls ++(0,-0.7) and ++(0,-0.7) .. (0.5,0.8);
\end{tikzpicture}
\displaybreak[1] \\
\intertext{and}
\begin{tikzpicture}[baseline=-1ex]
\node[fill=white, draw, rectangle] (left) at (0,0) {$\JW{2}$};
\begin{pgfonlayer}{background} 
\path[fill=black!20] (-0.75,-0.8) rectangle (0.75,0.8);
\path[fill=white]
	($(left.240)+(0,-0.5)$) -- ($(left.120)+(0,0.5)$) -- ($(left.60)+(0,0.5)$) -- ($(left.300)+(0,-0.5)$);
\draw ($(left.240)+(0,-0.5)$) -- ($(left.120)+(0,0.5)$);
\draw  ($(left.60)+(0,0.5)$) -- ($(left.300)+(0,-0.5)$);
\end{pgfonlayer}
\end{tikzpicture}
& = 
\begin{tikzpicture}[baseline=-1ex]
\path[fill=black!20] (-0.75,-0.8) rectangle (0.75,0.8);
\path[fill=white]
	 (-0.5,-0.8) -- (-0.5,0.8) -- (0.5,0.8) -- (0.5,-0.8);
\draw (-0.5,-0.8) -- (-0.5,0.8);
\draw  (0.5,0.8) -- (0.5,-0.8);
\end{tikzpicture}
 -  \frac{1}{[2]^2}
\begin{tikzpicture}[baseline=-1ex]
\path[fill=black!20] (-0.75,-0.8) rectangle (0.75,0.8);
\path[draw, fill=white]
	 (-0.5,-0.8) .. controls ++(0,0.7) and ++(0,0.7) .. (0.5,-0.8);
\path[draw, fill=white]
	 (-0.5,0.8) .. controls ++(0,-0.7) and ++(0,-0.7) .. (0.5,0.8);
\end{tikzpicture}
\end{align*}
\end{example}

\begin{example}
Consider an irreducible shaded planar algebra $\cP$ with index $D$. Choose $q$ so $q+q^{-1} = \sqrt{D}$. If $\cP$ is spherical then the dimensions of the Jones-Wenzl idempotents are given by
\begin{align*}
\trr{\JW{n} \in \cP_{n,\pm}} & = [n+1]_q \in \Rational(\sqrt{D}) \\
\intertext{while when $\cP$ is lopsided they are}
\trr{\JW{n} \in \cP_{n,+}} & =
	\begin{cases}
		\frac{[n+1]_q}{[2]_q} & \text{if $n$ is odd} \\
		[n+1]_q & \text{if $n$ is even}
	\end{cases}
	\\
	& \in \Rational(D) \displaybreak[1] \\
\trr{\JW{n} \in \cP_{n,-}} & =
	\begin{cases}
		[2]_q [n+1]_q & \text{if $n$ is odd} \\
		[n+1]_q & \text{if $n$ is even}
	\end{cases}
	\\
	& \in \Rational(D)
\end{align*}
\end{example}

\section{The shaded planar algebras $\cH_0$ and $\cH_1$}
\label{sec:construction}
%!TEX root = ../article.tex

The goal of this section is to prove Theorem \ref{thm:detailed} part \ref{pa:construct} and the ``if" direction of parts \ref{pa:dualprojs} and \ref{pa:principalprojs}.   Thus we explicitly construct certain rational forms and determine when they are complete.  These constructions all follow quickly from the generators and relations presentation of the `Haagerup subfactor planar algebra' $\cH_0$ (from \cite{0902.1294}) and of the `extended Haagerup subfactor planar algebra' $\cH_1$ (from \cite{0909.4099}).   We modify these previous constructions slightly by using the lopsided normalization for cups and caps instead of the spherical one.  Furthermore, we have rescaled the generator: our $S$ is $\lambda_\ell$ times the $S$ of \cite{0909.4099, 0902.1294}.  (Both of these modifications were made in order to simplify the numbers that appear in the definition so that they lie in a smaller field.)

Let $\ell=0$ or $1$, and let $n=4\ell+4$.  Choose $q$ so $q+q^{-1}=\sqrt{D_\ell}$, and define the quantum integers as usual by $[m]=\frac{q^m - q^{-m}}{q-q^{-1}} \in \Integer[\sqrt{D_\ell}]$ (so in particular, $[2] = \sqrt{D_\ell}$).  
If $m$ is odd then $[m] \in \Rational(D_\ell)$.  If $m$ is even, then $[2][m] = [m+1]+[m-1]\in \Rational(D_\ell)$, and also any ratio of two even quantum integers lies in $\Rational(D_\ell)$.  We will write $\check{r} = \frac{[n+2]}{[n]} \in \Rational(D_\ell)$ throughout. (This number is the ratio of Perron-Frobenius dimensions of the two vertices past the branch point on the dual principal graph. The corresponding ratio on the principal graph is just $1$.)

Define $\lambda_\ell = [2]^{-1}\sqrt{-\check{r}}$ (these are the same as the explicit numbers given in the introduction) and note that $\lambda_\ell^2 = - \frac{[n+2]}{[2]^2[n]} \in \Rational(D_\ell)$. 
%Further, \todo{Why is $d_\ell^2 \in \Rational(\lambda_\ell)$?} 
Since $\lambda_\ell^2$ is not rational, and the degrees $[\Rational(D_\ell):\Rational]$ are prime, it follows that $\Rational(D_\ell) = \Rational(\lambda_\ell^2)$, and in particular 
$\Rational(D_\ell) \subset \Rational(\lambda_\ell) \subset \Complex$. 

\begin{defn} \label{def:pa}
Let $k$ be any field $\Rational(D_\ell) \subset k \subset \Complex$.  Let $\cQ_\ell(k)$ be the shaded planar algebra over $k$ generated by a single $n$-box $S \in \cQ_\ell(k)_{n,+}$ subject to relations (1)-(6) below.
\begin{enumerate}
\item The lopsided moduli are $1$ and $D_\ell$.  That is, the shaded circle is $1$ and the unshaded circle is $D_\ell$.
\item $\rho(S)=-S$,
\item $S$ is uncappable,
\item $S^2 = \lambda_\ell^2 \JW{n}$,
\item \begin{align*}
	\scalebox{0.8}{%
%\beginpgfgraphicnamed{\pathtotrunk diagrams/tikz/#1-external}%
\begin{tikzpicture}[STrain]
	\RainbowOne;
        \draw (0,0)--(0,-0.5);
        \node[anchor=west] at (0,-0.35) {\footnotesize$2n+2$};
	\JWPlusTwo;
\end{tikzpicture}
%
%\endpgfgraphicnamed
}
	        & =  [2] [n]
	\scalebox{0.8}{%
%\beginpgfgraphicnamed{\pathtotrunk diagrams/tikz/#1-external}%
\begin{tikzpicture}[STrain]
	\STrainStrings{$n+1$}{$n+1$} \STrainOne
        \draw (0,0)--(0,-0.5);
        \node[anchor=west] at (0,-0.35) {\footnotesize$2n+2$};
	\JWPlusTwo
\end{tikzpicture}
%
%\endpgfgraphicnamed
},
	\end{align*}
\item and
	\begin{align*}
	\scalebox{0.8}{%
%\beginpgfgraphicnamed{\pathtotrunk diagrams/tikz/#1-external}%
\begin{tikzpicture}[STrain]
	\RainbowTwo
        \draw (0,0)--(0,-0.5);
        \node[anchor=west] at (0,-0.35) {\footnotesize$2n+4$};
	\JWPlusFour
\end{tikzpicture}
%
%\endpgfgraphicnamed
} 
	         = \frac{1}{\lambda_\ell^2} \frac{1}{[n+1]} \frac{[2n+4]}{[n+2]}
        \scalebox{0.8}{%
%\beginpgfgraphicnamed{\pathtotrunk diagrams/tikz/#1-external}%
\begin{tikzpicture}[STrain]
	\STrainThreeStrings{$n+1$}{$2$}{$n+1$} \STrainOneOne
        \draw (0,0)--(0,-0.5);
        \node[anchor=west] at (0,-0.35) {\footnotesize$2n+4$};
	\JWPlusFour
\end{tikzpicture}
%
%\endpgfgraphicnamed
}.
	\end{align*}
\end{enumerate}
\end{defn}

The planar algebra $\cQ_\ell(k)$ isn't nondegenerate, and hence doesn't yield semisimple tensor categories.  We apply the usual semisimplification procedure.

\begin{defn}
Let $\cH_\ell(k)$ be $\cQ_\ell(k) / \cN_\ell(k)$, where $\cN_\ell(k)$ is the ideal of negligible morphisms in $\cQ_\ell(k)$.
\end{defn}

\begin{thm}
$\cH_\ell(\Complex)$ is a unitary irreducible planar algebra with principal graph pair $\cH_\ell$.  Hence, the tensor categories $\cH_\ell^p(k)$ and $\cH_\ell^d(k)$ are fusion categories.
\end{thm}
\begin{proof}
The first statement follows immediately from \cite[Theorem 5.5]{0902.1294} (for $\ell=0$) and from \cite[Theorem 3.10]{0909.4099} (for $\ell=1$).  Unitarity implies that $\cH_\ell^p(\Complex)$ and $\cH_\ell^d(\Complex)$ are semisimple, and hence $\cH_\ell^p(k)$ and $\cH_\ell^d(k)$ are semisimple also (since nonsemisimplicity is preserved under base extension).  These categories are fusion because the principal graphs are finite.
\end{proof}

%The results of \cite{0902.1294} (for $\ell=0$) and of \cite{0909.4099} (for $\ell=1$) show that $\cH_\ell(\Complex)$ is non-zero, and a subfactor planar algebra. Since $\cH_\ell(\Complex) = \cH_\ell(k) \tensor_k \Complex$, this implies that $\cH_\ell(k)$ is also non-zero.

\begin{thm}
All the minimal idempotents in $\cH_\ell(\Complex)$ (corresponding to the vertices of the principal graph) in fact lie inside $\cH_\ell(\Rational(\lambda_\ell))$.  All of the minimal idempotents in the even part of the dual principal graph lie inside $\cH_\ell(\Rational(D_\ell))$.
\end{thm}

\begin{proof}
We give explicit formulas for all the idempotents, in Lemmas \ref{lem:idempotents1}, \ref{lem:idempotents2} and \ref{lem:idempotents3}, which occupy the rest of this section. Our naming conventions for the idempotents are shown in Figure \ref{fig:naming}.
\end{proof}
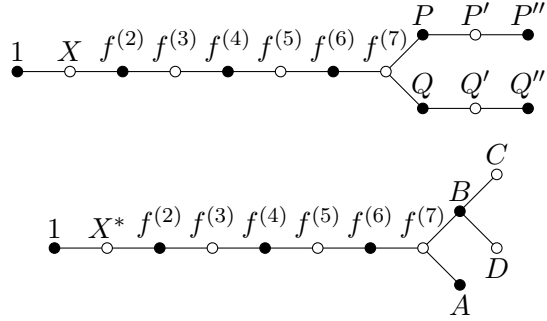
\begin{figure}[!htb]
$$\begin{tikzpicture}[baseline=-1mm, scale=.7]
	\draw (0,0) -- (7,0);
	\draw (7,0)--(7.7,.7)--(9.7,.7);
	\draw (7,0)--(7.7,-.7)--(9.7,-.7);

	\filldraw              (0,0) node [above] {$1$} circle (1mm);
	\filldraw [fill=white] (1,0) node [above] {$X$} circle (1mm);
	\filldraw              (2,0) node [above] {$\JW{2}$}  circle (1mm);
	\filldraw [fill=white] (3,0) node [above] {$\JW{3}$}  circle (1mm);
	\filldraw              (4,0) node [above] {$\JW{4}$}  circle (1mm);
	\filldraw [fill=white] (5,0) node [above] {$\JW{5}$}  circle (1mm);
	\filldraw              (6,0) node [above] {$\JW{6}$}  circle (1mm);
	\filldraw [fill=white] (7,0) node [above] {$\JW{7}$}  circle (1mm);

	\filldraw              (7.7,.7) node [above] {$P$}  circle (1mm);
	\filldraw [fill=white] (8.7,.7) node [above] {$P'$}  circle (1mm);
	\filldraw              (9.7,.7) node [above] {$P''$}  circle (1mm);
	\filldraw              (7.7,-.7) node [above] {$Q$}  circle (1mm);
	\filldraw [fill=white] (8.7,-.7) node [above] {$Q'$}  circle (1mm);
	\filldraw              (9.7,-.7) node [above] {$Q''$}  circle (1mm);
\end{tikzpicture}$$

$$\begin{tikzpicture}[baseline=-1mm, scale=.7]
	\draw (0,0) -- (7,0);
	\draw (7,0)--(7.7,.7);
	\draw (7,0)--(7.7,-.7);
	\draw (7.7,.7)--(8.4,1.4);
	\draw (7.7,.7)--(8.4,0);

	\filldraw              (0,0) node [above] {$1$} circle (1mm);
	\filldraw [fill=white] (1,0) node [above] {$X^*$} circle (1mm);
	\filldraw              (2,0) node [above] {$\JW{2}$}  circle (1mm);
	\filldraw [fill=white] (3,0) node [above] {$\JW{3}$}  circle (1mm);
	\filldraw              (4,0) node [above] {$\JW{4}$}  circle (1mm);
	\filldraw [fill=white] (5,0) node [above] {$\JW{5}$}  circle (1mm);
	\filldraw              (6,0) node [above] {$\JW{6}$}  circle (1mm);
	\filldraw [fill=white] (7,0) node [above] {$\JW{7}$}  circle (1mm);

	\filldraw              (7.7,.7) node [above] {$B$}  circle (1mm);
	\filldraw              (7.7,-.7) node [below] {$A$}  circle (1mm);
	\filldraw [fill=white] (8.4,1.4) node [above] {$C$}  circle (1mm);
	\filldraw [fill=white] (8.4,0) node [below] {$D$}  circle (1mm);
\end{tikzpicture}$$
\caption{The names for the minimal idempotents in $\cH_1(\Complex)$. Even vertices are shown as filled circles, and odd vertices are shown as hollow circles.  The convention for $\cH_0(\Complex)$ is the same but with fewer $\JW{i}$.}
\label{fig:naming}
\end{figure}

\begin{lem}
\label{lem:S2}%
The `one-click' rotation $\rho^{1/2}(S)$, which is an element of $\cH_\ell(k)_{n,-}$, satisfies
\begin{align}
\rho^{1/2}(S)^2 & = -i [2] \lambda_\ell (\check{r}^{1/2} - \check{r}^{-1/2}) \rho^{1/2}(S) - [2]^2 \lambda_\ell^2 \JW{n} \notag \\
	& =  (\check{r}-1) \rho^{1/2}(S) + \check{r} \JW{n}. \label{eq:rhoS2}
\end{align}
\end{lem}
\begin{proof}
Take the formula for $\rho^{1/2}(S)^2$ appearing in \cite[Theorem 3.9]{0909.4099},
% $$\rho^{1/2}(S)^2 = - \omega^{1/2} r^{1/2} (\check{r}^{1/2} - \check{r}^{-1/2}) \rho^{1/2}(S) + \omega^{-1} r \JW{n}.$$
and recall that the $S$ there is $\lambda_\ell^{-1}$ times our generator $S$, and that since we are working in the lopsided planar algebra, there is an additional factor of $[2]^{-1}$ with each occurrence of $\rho^{1/2}$. This factor arises because in the notation of \S \ref{sec:lopsided}, if the planar algebras from \cite{0902.1294, 0909.4099} are denoted $\cP$, we're looking at $\cP_{\sqrt{D_\ell}}$, and $\rho^{1/2}$ has a single critical point shaded above, which is a maximum.
 Thus we obtain
$$[2]^{-2} \lambda_\ell^{-2} \rho^{1/2}(S)^2 = - [2]^{-1} \lambda_\ell^{-1} \omega^{1/2} r^{1/2} (\check{r}^{1/2} - \check{r}^{-1/2})\rho^{1/2}(S) + \omega^{-1} r \JW{n}$$
and after setting $r=1$, $\omega = -1$, using the formula for $\lambda_\ell$ and rearranging, this becomes the desired formula.
% \begin{align*}
% $$\rho^{1/2}(S)^2 = - i [2] \lambda_\ell (\check{r}^{1/2} - \check{r}^{-1/2})\rho^{1/2}(S) - [2]^{2} \lambda_\ell^{2}  \JW{n}.$$
%\end{align*}
\end{proof}

\begin{lem}
\label{lem:idempotents1}
The idempotents $P$, $Q$, $A$ and $B$ in $\cH_\ell^p$ and $\cH_\ell^d$ are given by 
\begin{align}
P & = \frac{1}{2}\left(\JW{n} + \frac{1}{\lambda_\ell}S\right) \label{eq:P} \\
Q & = \frac{1}{2}\left(\JW{n} - \frac{1}{\lambda_\ell}S\right) \label{eq:Q} \\
A & = \frac{1}{\check{r}+1}\left(\JW{n}  + \rho^{1/2}(S)\right) \label{eq:A} \\
B & = \frac{1}{\check{r}+1}\left(\check{r} \JW{n} - \rho^{1/2}(S)\right) \label{eq:B}
\end{align}
\end{lem}
\begin{proof}
The formulas certainly define idempotents, according to the formulas in Lemma \ref{lem:S2} and in Definition \ref{def:pa}. Since $\frac{1}{2}\left(\JW{n} + \frac{1}{\lambda_\ell}S\right)+\frac{1}{2}\left(\JW{n} - \frac{1}{\lambda_\ell}S\right) = \JW{n}$, we know these formulas for $P$ and $Q$ actually define the corresponding idempotents. Note the symmetry between $P$ and $Q$: rescaling $S$ by $-1$ interchanges the formulas for the idempotents.
Since $\frac{1}{\check{r}+1}\left(\JW{n}  + \rho^{1/2}(S)\right) + \frac{1}{\check{r}+1}\left(\check{r} \JW{n} - \rho^{1/2}(S)\right) = \JW{n}$, these formulas are indeed the idempotents $A$ and $B$; you can check which is which by taking the trace.
\end{proof}

\begin{lem}
\label{lem:idempotents2}
The remaining idempotents on the principal graph are given by 
\begin{align}
P' & = P \tensor X - \frac{2[n]}{[2][n+1]} \left((P \tensor X) \circ e_n \circ (P \tensor X)\right) \label{eq:P'} \displaybreak[1]\\
Q' & = Q \tensor X - \frac{2[n]}{[2][n+1]} \left((Q \tensor X) \circ e_n \circ (Q \tensor X)\right) \label{eq:Q'} \displaybreak[1] \\
P'' & = P' \tensor X^* -  \frac{[2][n+1]}{[n+2]-[n]} \left((P' \tensor X^*) \circ e_{n+1} \circ (P' \tensor X^*)\right) \label{eq:P''} \displaybreak[1] \\
\intertext{and}
Q'' & = Q' \tensor X^* -  \frac{[2][n+1]}{[n+2]-[n]} \left((Q' \tensor X^*) \circ e_{n+1} \circ (Q' \tensor X^*)\right).  \label{eq:Q''} 
\end{align}
\end{lem}
\begin{proof}
We know from the principal graph that $P \tensor X \iso \JW{n-1} \directSum P'$
We claim that the second term in the right hand side of Equation \eqref{eq:P'} is the projection onto $\JW{n-1}$, and so the difference is exactly $P'$. In fact, if the second term is a projection at all it must be equivalent to $\JW{n-1}$ as it factors through $n-1$ strands. The partial trace of $P$ is 
$$\ptr{P} = \frac{[2][n+1]}{2[n]} \JW{n-1}$$
and an easy calculation gives the result. An identical argument for $Q$ instead of $P$ gives Equation \eqref{eq:Q'}.

Equations \eqref{eq:P''} and \eqref{eq:Q''} follow similarly. We use that fact that $P' \tensor X^* \iso P \directSum P''$, and calculate $$\ptr{P'} = \left(1-\frac{2[n]}{[2][n+1]}\right) P = \frac{[n+2]-[n]}{[2][n+1]} P$$ from Equation \eqref{eq:P'}. Computing the square of the second term in Equation \eqref{eq:P''} confirms that it is a projection, and gives the result. The analogous considerations for $Q'$ instead of $P'$ give Equation \eqref{eq:Q''}.
\end{proof}

\begin{lem}
\label{lem:idempotents3}%
The idempotents $C$ and $D$ on the dual principal graph are given by
\begin{align}
C & = \rho^{n/2}(P') \label{eq:C} \\
\intertext{and}
D & = \rho^{n/2}(Q'). \label{eq:D}
\end{align}
\end{lem}
\begin{proof}
The idempotents $C$ and $P'$ are dual, as are $D$ and $Q'$.
\end{proof}

The coefficients appearing in Equations \eqref{eq:P} and \eqref{eq:Q} lie in $\Rational(\lambda_\ell)$ while the coefficients in Equations \eqref{eq:A} through \eqref{eq:D} all lie in $\Rational(D_\ell)$.

This planar algebra gives a possibly incomplete rational form for both even parts over $\Rational(D_\ell)$.  Since all of the projections in the dual even part lie in $\Rational(D_\ell)$ we have constructed a complete rational form for the dual even part in $\Rational(D_\ell)$.  Finally, since all of the projections in the principal even part lie in $\Rational(\lambda_\ell)$ we have a complete rational form for the principal even part over $\Rational(\lambda_\ell)$.

\section{Rational forms and twisted moments}
\label{sec:noncyclotomic}
%!TEX root = ../article.tex

The goal of this section is to prove the ``only if" direction of Thereom \ref{thm:detailed} parts \ref{pa:dualprojs} and \ref{pa:principalprojs} and as an immediate corollary Theorem \ref{thm:noncyclotomic}.

First we prove part \ref{pa:dualprojs}.  Suppose $\cH_k^d$ is a complete rational from of $\cH^d$.  The absolute value squared of the dimension of any object in $\cH^d$ must lie in $k$, by \cite[\S 2.1]{MR2183279} (we take the absolute value squared because this does not depend on having a pivotal structure over $k$).  Hence $|D_\ell |^2 \in k$, but $\Rational(D_\ell) = \Rational(|D_\ell^2|)$ in our cases.

Now we turn our attention to the main result, part \ref{pa:principalprojs}.  Again our approach is to find an invariant which does not depend on the rational form and see that it generates $\Rational(\lambda_\ell)$.  A key observation in \cite{0902.1294, 0909.4099} is that the whole structure of the Haagerup and extended Haagerup planar algebras can be recovered from the ``moments" and ``twisted moments" of a generator $S$.  The moments are just the diagrammatic traces of the powers of $S$, while the twisted moments are the diagrammatic traces of the powers of $\rho^{\frac{1}{2}}(S)$ where $\rho^{\frac{1}{2}}$ is the `one-click' rotation.  Since $\rho^{\frac{1}{2}}$ changes the shading, it only makes sense in the shaded planar algebra, rather than the even part.  Nonetheless we show in this section how to construct normalized twisted moments of an element $R$ in a shaded planar algebra which are invariants of the even part.  In particular, if $\cC_K$ is the even part of a $3$-supertransitive shaded planar algebra, $R$ is a $\Rational$-linear combination of projections in $\cC_K$ and if $\cC_k$ is a complete rational form of $\cC_K$ over $k$, then $k$ must contain the normalized third twisted moment of $R$.

In order to define the twisted moments  we need to define several morphisms.  Suppose that $\cC_k$ is a split fusion category over $k$ with an object $Y$ such that $\Hom{}{Y\tensor Y \tensor Y}{1}$ is $1$-dimensional (this is assured if $\cK_k$ is the even part of a 3-supertransitive planar algebra) and $R$ is a $\Rational$-linear combination of simple projections in $\End{Y^{\otimes m}}$ for some $m$.  Later, we will specialize to the category $\cC_k = \cH_\ell^p$, the object $Y=\JW{2}$, and the morphism $$R = \frac{1}{\lambda_\ell} S = \frac{1}{2}(P-Q) \in \End{{\JW{2}}^{\tensor n/2}}.$$  Choose nonzero elements $B$, $B'$, $T$, and $T'$ in the $1$-dimensional spaces $\Hom{}{Y \otimes Y}{1}$, $\Hom{}{1}{Y \otimes Y}$, $\Hom{}{Y \otimes Y \otimes Y}{1}$, and $\Hom{}{1}{Y \otimes Y \otimes Y}$ respectively.  We can build morphisms $B_{Y^{\otimes 2i}}: Y^{\otimes 2i} \rightarrow 1$ and $B_{Y^{\otimes 2i}}: 1 \rightarrow Y^{\otimes 2i}$  by composing appropriately many copies of $B$ or $B'$.  For example, $B_{Y^{\otimes 4}} =  B \circ (1 \otimes B \otimes 1)$.

\begin{defn}
Define the theta symbol $\Theta\in k$ by $T \circ T'$, and the circle $C \in k$ by $B \circ B'$.  Using the usual string notation, with a string labeled by $Y$, we have:
$$\begin{array}{ccc}\Theta(T, T') =%
%\beginpgfgraphicnamed{\pathtotrunk diagrams/tikz/#1-external}%
\begin{tikzpicture}[baseline=2ex]
\path (0,0) node (T') [shape=rectangle, draw]{\;$T'$\;};
\path (0,1) node (T) [shape=rectangle, draw]{\;$T$\;};

\draw [->] (T'.140) -- (T.220);
\draw [->] (T'.90) -- (T.270);
\draw [->] (T'.40) -- (T.320);
\end{tikzpicture}%
%\endpgfgraphicnamed
 & \text{and} & C(B, B') = %
%\beginpgfgraphicnamed{\pathtotrunk diagrams/tikz/#1-external}%
\begin{tikzpicture}[baseline=2ex]
\path (0,0) node (B') [shape=rectangle, draw]{\;$B'$\;};
\path (0,1) node (B) [shape=rectangle, draw]{\;$B$\;};

\draw [->] (B'.130) -- (B.230);
\draw [->] (B'.50) -- (B.310);
\end{tikzpicture}%
%\endpgfgraphicnamed
\end{array}$$
\end{defn}

Let a thick line denote $m-1$ copies of a thin line (i.e. the object $Y^{\otimes m-1}$).  
%Notice that,$$B_{Y^{\otimes m}} \circ B'_{Y^{\otimes m}} =  \inputtikz{circles} = C^m.$$

\begin{defn}
Define $\widehat{M}_3(R; B, B', T, T') \in k$, the unnormalized third twisted moment of a morphism $R \in \End{Y^{\otimes m}}$, by the value of the following composition
\begin{align*}
\widehat{M}_3(R; B, B', T, T') & =  
	B_{Y^{\otimes 2m-2}} \circ (\mathrm{id}_{Y^{\otimes m-1}} \otimes T \otimes \mathrm{id}_{Y^{\otimes m-1}}   ) \circ \\
	&\qquad \circ (R \tensor \mathrm{id}_{Y^{\otimes m+1}}) \circ  (\mathrm{id}_{Y} \otimes R \otimes \mathrm{id}_{Y^{\otimes m}}   ) \circ \\ & \qquad \qquad\circ (\mathrm{id}_{Y^{\otimes 2}} \otimes R \otimes \mathrm{id}_{Y^{\otimes m-1}}   ) \circ (T' \tensor B'_{Y^{\otimes 2m-2}}) \displaybreak[1]
\\ 
& = 
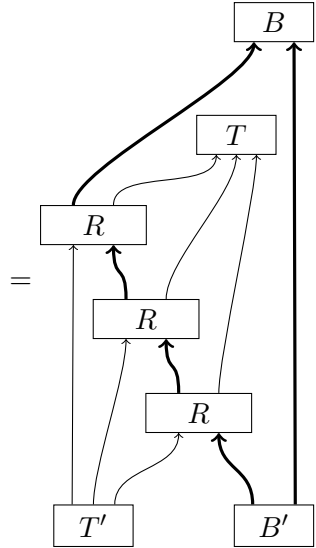
\begin{tikzpicture}
[
	baseline=10ex,
	R/.style={rectangle,draw,minimum width=40pt},
	O/.style={rectangle,draw,minimum width=30pt},
]
% place the rectangles with R
\foreach \i / \xc / \yc in {1/0/0, 2/-0.7/1.25, 3/-1.4/2.5}
{
	\node[R](R\i) at (\xc,\yc)  {$R$};
}
% place B, B', T and T'
\begin{scope}[every node/.style={O}]
	\node(B) at (1,5.2){$B$};
	\node(B') at (1,-1.5) {$B'$};
	\node(T) at (0.5,3.7) {$T$};
	\node(T') at (-1.4,-1.5) {$T'$};
\end{scope}
%thick lines from R to R
\newcommand{\upwardsarrow}[3][]{
 \draw[->,#1] #2 ..controls ++(0,0.5) and ++(0,-0.5) .. #3;
}
\foreach \i / \j in {1/2,2/3}
{
	\upwardsarrow[very thick] {(R\i.135)} {(R\j.-45)}
	%\draw[->, very thick] (R\i.45) .. controls ++(0,1) and ++(0,-1) .. (R\j.225);
}
%thick lines into and out of B and B'
\upwardsarrow[very thick] {(B'.45)}{(B.-45)};
\upwardsarrow[very thick] {(R3.135)}{(B.225)};
\upwardsarrow[very thick] {(B'.135)}{(R1.-45)};
%lines into T
\foreach \i / \theta in {1/-45,2/-90,3/-135}
{
	\upwardsarrow {(R\i.45)} {(T.\theta)};
}
%lines out of T'
\foreach \i / \theta in {1/45,2/90,3/135}
{
	\upwardsarrow {(T'.\theta)} {(R\i.-135)};
}
\end{tikzpicture}
\end{align*}

Define the normalized third twisted moment by $$M_3(R) = \frac{\widehat{M}_3(R; B, B', T, T')}{\Theta(T, T') C(B, B')^{m-1}}.$$
\end{defn}

\begin{lem}
$M_3(R)$ does not depend on the choice of $B$, $B'$, $T$, and $T'$.
\end{lem}
\begin{proof}
Each of $B$, $B'$, $T$, and $T'$ are well-defined up to a choice of scalar.  Each of $\Theta$, $C$, and $\widehat{M}_3(R)$ are homogoenous with respect to these rescalings, and it is easy to see that $M_3(R)$ is degree $0$ with respect to each of these rescalings.
\end{proof}

\begin{thm}
If $\cC_k$ is a split rational form for $\cH_\ell^p(\Complex)$ over a field $k$, then $\lambda_\ell \in k$.
\end{thm}
\begin{proof}
Let $Y$ be the simple object of $\cC_k$ corresponding to the object $\JW{2}$ of $\cH_\ell^p(\Complex)$.   Let $R = \frac{1}{2}(P-Q) = \frac{1}{\lambda_\ell} S \in \End{Y^{\otimes 2 \ell + 2}}$.  The morphism $R$ makes sense in $\cC_k$ because $\cC_k$ is split.  Consider the scalar $M_3(R) \in k$.  Since this scalar does not depend on any choices, we  can compute it in $\cH_\ell^p(\Complex)$ using the following obvious choices:
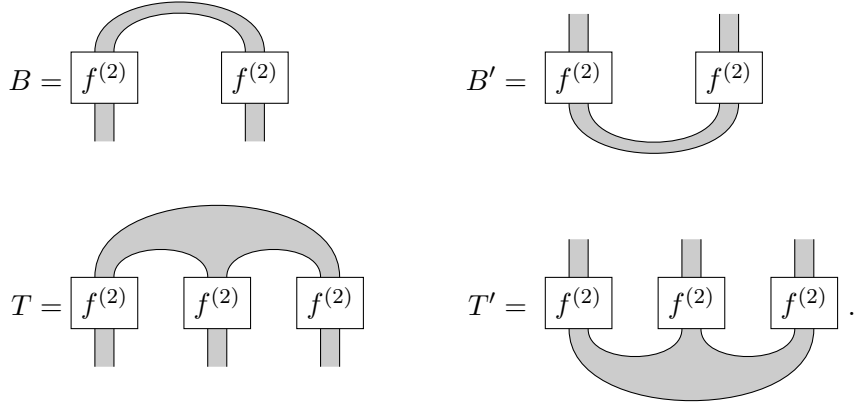
\begin{align*}
B & =%
%\beginpgfgraphicnamed{\pathtotrunk diagrams/tikz/#1-external}%
\begin{tikzpicture}[baseline=-1ex]
\node[fill=white, draw, rectangle] (left) at (0,0) {$\JW{2}$};
\node[fill=white, draw, rectangle] (right) at (2,0) {$\JW{2}$};

\begin{pgfonlayer}{background} 
\draw[fill=black!20]
	($(left.250)+(0,-0.5)$) -- (left.110) to [bend left=90] (right.70) -- ($(right.290)+(0,-0.5)$)
	($(right.250)+(0,-0.5)$) -- (right.110) to [bend right=90] (left.70) -- ($(left.290)+(0,-0.5)$);
\end{pgfonlayer}

\end{tikzpicture}%
%\endpgfgraphicnamed
 & B' & = %
%\beginpgfgraphicnamed{\pathtotrunk diagrams/tikz/#1-external}%
\rotatebox{180}{
\begin{tikzpicture}[baseline=1ex]
\node[fill=white, draw, rectangle] (left) at (0,0) {\rotatebox{180}{$\JW{2}$}};
\node[fill=white, draw, rectangle] (right) at (2,0) {\rotatebox{180}{$\JW{2}$}};
\begin{pgfonlayer}{background} 
\draw[fill=black!20]
	($(left.250)+(0,-0.5)$) -- (left.110) to [bend left=90] (right.70) -- ($(right.290)+(0,-0.5)$)
	($(right.250)+(0,-0.5)$) -- (right.110) to [bend right=90] (left.70) -- ($(left.290)+(0,-0.5)$);
\end{pgfonlayer}
\end{tikzpicture}
}%
%\endpgfgraphicnamed
 \displaybreak[1]\\
T & =%
%\beginpgfgraphicnamed{\pathtotrunk diagrams/tikz/#1-external}%
\begin{tikzpicture}[baseline=-1ex]
\foreach \x / \label in {0/left, 1.5/middle, 3/right} {
	\node(\label)[fill=white,draw,rectangle] at (\x,0) {$\JW{2}$};
}
% now the shading
\begin{pgfonlayer}{background} 
\draw[fill=black!20]
	($(left.250)+(0,-0.5)$) -- (left.110) to [bend left=90] (right.70) -- ($(right.290)+(0,-0.5)$)
	($(right.250)+(0,-0.5)$) -- (right.110) to [bend right=90] (middle.70) -- ($(middle.290)+(0,-0.5)$)
	($(middle.250)+(0,-0.5)$) -- (middle.110) to [bend right=90] (left.70) -- ($(left.290)+(0,-0.5)$);
\end{pgfonlayer}
\end{tikzpicture}%
%\endpgfgraphicnamed
 & T' & = %
%\beginpgfgraphicnamed{\pathtotrunk diagrams/tikz/#1-external}%
\rotatebox{180}{
\begin{tikzpicture}[baseline=1ex]
\foreach \x / \label in {0/left, 1.5/middle, 3/right} {
	\node(\label)[fill=white,draw,rectangle] at (\x,0) {\rotatebox{180}{$\JW{2}$}};
}
% now the shading
\begin{pgfonlayer}{background} 
\draw[fill=black!20]
	($(left.250)+(0,-0.5)$) -- (left.110) to [bend left=90] (right.70) -- ($(right.290)+(0,-0.5)$)
	($(right.250)+(0,-0.5)$) -- (right.110) to [bend right=90] (middle.70) -- ($(middle.290)+(0,-0.5)$)
	($(middle.250)+(0,-0.5)$) -- (middle.110) to [bend right=90] (left.70) -- ($(left.290)+(0,-0.5)$);
\end{pgfonlayer}
\end{tikzpicture}
}%
%\endpgfgraphicnamed
. 
\end{align*}
Since $S$ is uncappable, any $\JW{2}$ connected to an $S$ can be replaced with the identity.  Thus with the above choices the unnormalized third twisted moment $\widehat{M}_3(R)$ is given by
\begin{align*}\widehat{M}_3(R) =
\begin{tikzpicture}
[
	baseline=10ex,
	R/.style={rectangle,draw,minimum width=30pt},
]
\clip (2,-1.2) rectangle (-2.2,4.5);
% place the rectangles with R
\foreach \i / \xc / \yc in {1/0/0, 2/-0.7/1.25, 3/-1.4/2.5}
{
	\node[R](R\i) at (\xc,\yc)  {$R$};
	\foreach \n/\q in {1/30, 2/70, 3/110, 4/150} {
		\coordinate(R\i\n) at (R\i.\q);
	}
	\foreach \n/\q in {5/210, 6/250, 7/290, 8/330} {
		\coordinate(R\i\n) at (R\i.\q);
	}
}
\draw[fill=black!20] (R14) -- (R27) -- (R28) -- (R13);
\draw[fill=black!20] (R24) -- (R37) -- (R38) -- (R23);
\newcommand{\rr}{0.8}
\newcommand{\oa}{2.5}
\newcommand{\ob}{0.6}
\newcommand{\stf}{1.4}
\draw[fill=black!20] (R34) .. controls ++ ($2*\stf*(0,\rr)$) and ++ ($2*\stf*(0,\rr)$) .. ($(R34)+(\oa,0)+(0.8,0)$) -- ($(R17)+(\ob,0)+(0.8,0)$) .. controls ++ ($-\stf*(0,\rr)$) and ++ ($-\stf*(0,\rr)$)  .. (R17) -- (R18) .. controls ++ (0,-\rr)  and ++ (0,-\rr) .. ($(R18)+(\ob,0)$) -- ($(R33)+(\oa,0)$)  .. controls ++ (0,2*\rr) and ++(0,2*\rr) .. (R33);
\draw[fill=black!20] (R32) .. controls ++ ($\stf*(0,\rr)$) and ++ ($4*\stf*(0,\rr)$) .. (R11) -- (R12)  .. controls ++ (0,2*\rr) and ++ (0,\rr) .. (R21) -- (R22)  .. controls ++ (0,2*\rr) and ++ (0,\rr) .. (R31);
\draw[fill=black!20] (R16) .. controls ++ ($\stf*(0,-\rr)$) and ++ ($4*\stf*(0,-\rr)$) .. (R35) -- (R36)  .. controls ++ (0,-2*\rr) and ++ (0,-\rr) .. (R25) -- (R26)  .. controls ++ (0,-2*\rr) and ++ (0,-\rr) .. (R15);
	\node at ($(R1.180)+(-0.1,0)$) {$\star$};
	\node at ($(R2.180)+(-0.1,0)$) {$\star$};
	\node at ($(R3.180)+(-0.1,0)$) {$\star$};
\end{tikzpicture}
& = \frac{1}{[2]^2 \lambda_\ell^3} %
%\beginpgfgraphicnamed{\pathtotrunk diagrams/tikz/#1-external}%
%!TEX root = ../../article.tex
\begin{tikzpicture}[traces]
\begin{scope}[yshift=-3cm]

\begin{pgfonlayer}{background} 
	\fill[black!20] (-3,-10) rectangle (7,16);
\end{pgfonlayer} 

\clip (-3,-10) rectangle (7,16);

\begin{scope}
	\filldraw[unshaded] (0,3) .. controls ++(-157:3cm) and ++(157:3cm) .. (0,-3) .. controls ++(112:2cm) and ++(-112:2cm) .. (0,3);
	\filldraw[unshaded] (0,3) .. controls ++(-22:3cm) and ++(22:3cm) .. (0,-3) .. controls ++(67:2cm) and ++(-67:2cm) .. (0,3);

\begin{scope}[yshift=6cm]
	\filldraw[unshaded] (0,3) .. controls ++(-157:3cm) and ++(157:3cm) .. (0,-3) .. controls ++(112:2cm) and ++(-112:2cm) .. (0,3);
	\filldraw[unshaded] (0,3) .. controls ++(-22:3cm) and ++(22:3cm) .. (0,-3) .. controls ++(67:2cm) and ++(-67:2cm) .. (0,3);
\end{scope}
	
	\filldraw[unshaded] (0,9) .. controls ++(22:3cm) and ++(90:4cm) .. (3,3) .. controls ++(-90:4cm) and ++(-22:3cm) .. (0,-3) .. controls ++(-67:6cm) and ++(-90:5cm) .. (4,3) .. controls ++(90:5cm) and ++(67:6cm) .. (0,9);
	\filldraw[unshaded] (0,9) .. controls ++(112:4cm) and ++(90:15cm) .. (5,3) .. controls ++(-90:15cm) and ++(-112:4cm) .. (0,-3) .. controls ++(-157:7cm) and ++(-90:20cm) .. (6,3) .. controls ++(90:20cm) and ++(157:7cm) .. (0,9);
	
	\node at (0,3) [Tbox] (T1) {$S$};
	\node at (0,9) [Tbox] (T3) {$S$};
	\node at (0,-3) [Tbox] (T2) {$S$};
	\node at (T1.100) [above left] {$\star$};
	\node at (T2.100) [above left] {$\star$};
	\node at (T3.90) [above left] {$\star$};

\end{scope}
\end{scope}
\end{tikzpicture}%
%\endpgfgraphicnamed
 \\
& = \frac{1}{[2]^2 \lambda_\ell^3} \tr{\rho^{\frac{1}{2}}(S)^3}
\end{align*}

%the diagram (we show the $\ell = 0$ example for simplicity)
%\todo{the old diagram was very wrong}
%which is equal to 
hence the name ``third twisted moment." (In the equation above, we don't literally mean equality of these diagrams; they live in different vector spaces. Instead, we mean that each diagram is the same multiple of the appropriate empty diagram.) The value of this moment can be easily computed using the formula for the square of $\rho^{\frac{1}{2}}(S)$ in Equation \eqref{eq:rhoS2}:
\begin{align*}
 \frac{1}{[2]^2 \lambda_\ell^3} \tr{\rho^{\frac{1}{2}}(S)^3} 
	& = \frac{1}{[2]^2 \lambda_\ell^3} \tr{(\check{r}-1) \rho^{1/2}(S)^2 + \check{r} \rho^{1/2}(S)} \displaybreak[1] \\
	& = \frac{1}{[2]^2 \lambda_\ell^3} \check{r} (\check{r}-1) [n+1]  \displaybreak[1] \\
	& = \lambda_\ell [2]^2(1-\check{r}^{-1})[n+1] \displaybreak[1] \\
	& = \lambda_\ell [2]^2 \frac{[2n+2]}{[n+2]}.
\end{align*}
(In the last step we used the identity $([n+2]-[n])[n+1] = [2n+2]$.)

We have
$\Theta  = [5] + 1$ and $C  = [3]$
and thus
\begin{align*}
M_3(R) & =\lambda_\ell  \frac{[2n+2]}{[n+2]} \frac{[2]^2}{\left([5]+1\right) [3]^{2\ell+1}}.
\end{align*}

Finally, since $M_3(R) \in k$, and $\frac{[2n+2]}{[n+2]} \frac{[2]^2}{\left([5]+1\right) [3]^{2\ell+1}} \in k$ since odd quantum numbers lie in $\Rational(D_\ell)$, $[2]^2 = [3]+1$ can be written as a sum of odd quantum numbers, and ratios of even quantum numbers (in particular $\frac{[2n+2]}{[n+2]}$) lie in $\Rational(D_\ell)$, we conclude that $\lambda_\ell \in k$.
\end{proof}

\section{The Drinfel'd center}
\label{sec:center}
%!TEX root = ../article.tex

\newcommand{\iHom}{\underline{\operatorname{Hom}}}
\newcommand{\actson}{\mathrel{\reflectbox{\rotatebox[origin=c]{90}{$\circlearrowleft $}}}\;}

The goal of this section is to give a proof of part \ref{pa:center} of Theorem \ref{thm:detailed}, completing the proof of Theorem \ref{thm:detailed}.   Our proof is somewhat unsatisfying since it is indirect and relies on extensive computations from \cite{MR1832764}.  Presumably one could use calculations along the lines of \cite{MR1832764} to give a more direct proof (specifically by explicitly writing down formulas for the half-braidings which only use scalars in $\Rational(\zeta_{39})$).

%Furthermore, we assume here that any fusion category can be realized as the category of representations of a weak Hopf algebra \cite{MR1976459, 0206113, 9904073, MR2522429}, although the most general result in the literature only says that split fusion categories over an arbitrary field of characteristic zero can be realized as the category of representations of a split semisimple weak Hopf algebra.  We fully expect that the argument in \cite{MR2522429} holds relaxing splitness, but the results of this section should be viewed more skeptically than those in the rest of the paper.

First notice that the entries of the $T$-matrix (see below) for $Z(\cH_0^p(\Complex))$ generate $\Rational(\zeta_{39})$ so if there is a complete rational form over a field $k$ as a ribbon category, then $k$ must contain $\zeta_{39}$.

Now we prove the other direction, namely that $Z(\cH_0^p(\Complex))$ has a complete rational form over $\Rational(\zeta_{39})$ as a ribbon category.  Explicitly we will prove that $Z(\cH_0^p(\Rational(\zeta_{39})))$ is a complete rational form.  The argument will take place in two steps. First we prove a general result that the Drinfel'd center of a rational form is a rational form for the Drinfel'd center. Next, we check that $Z(\cH_0^p(\Rational(\zeta_{39})))$ is split.  For both steps of the argument the key element is the induction functor $I: \cC \rightarrow Z(\cC)$ which commutes with base extension.

We quickly recall some key definitions and results concerning the Drinfel'd center and module categories (see \cite{MR1976459,MR1976233}).  Let $\cD$ be a fusion category and $\cM$ a semisimple finite left module category.  If $M_1$ and $M_2$ are objects in $\cM$ then the internal hom $\iHom_{\cD \actson \cM}(M_1,M_2)$ is defined to be the object in $\cD$ which represents the functor $X \mapsto \Hom{}{X \otimes M_1}{M_2}$.  Define the dual category $\cD^* = \mathrm{Fun}_{\cD}(\cM,\cM)$ of module endofunctors of $\cM$.  If $M$ is a simple object in $\cM$, there is a natural algebra structure on the object $A = \iHom_{\cD \actson \cM}(M,M)$ and the categories $\cM$, $\cM^{op}$, and $\cD^*$ can be concretely identified with the categories of right $A$-modules in $\cM$, left $A$-modules in $\cM$, and $A$--$A$ bimodules in $\cM$, respectively.  Fixing the simple object $M$, there's an induction functor $I: \cM \rightarrow \cD^*$, given by tensoring a right $A$-module on the left with $A$.

The Drinfel'd center $Z(\cC)$ has two equivalent definitions.  First, the objects of $Z(\cC)$ are pairs $(X, \{\sigma_{X,Y}\}_Y)$ consisting of an object in $\cC$ and a collection of half-braidings $\sigma_{X,Y}: X \otimes Y \rightarrow Y \otimes X$ satisfying certain naturality conditions.  For the second, recall that $\cC$ is a module category over $\cC \boxtimes \cC^{op}$  (with the action $(X\boxtimes Y) \otimes M \rightarrow X \otimes M \otimes Y$).  Then, $Z(\cC)$ is the dual of $\cC \boxtimes C^{op}$.  Explicitly, we can take $M = 1$, and let $A = \iHom_{\cC \boxtimes \cC^{op} \actson \cC}(1,1)$ be the internal endomorphisms of $1 \in \cC$.  Thus $A$ is an algebra object in $\cC \boxtimes \cC^{op}$, and $Z(C)$ is the category of $A$--$A$ bimodule objects in $\cC \boxtimes \cC^{op}$.  From this point of view the induction functor $I: \cC \rightarrow Z(\cC)$ is given by $X \mapsto A \otimes X$.

\begin{rem}
It is natural to wonder what $I(X)$ is in terms of the first definition of the center.  According to \cite[Prop. 5.4]{MR2183279} for a split fusion category the underlying object of $I(X)$ is $\bigoplus_V V \otimes X \otimes V^*$ (where the sum is taken over all simples).  This formula certainly does not work in the non-split case.  Notice, however, that an analogous problem occurs when decomposing the adjoint representation of the group ring $k[G]$.  Here, in the nonsplit case, the adjoint representation is not $\bigoplus_V V \otimes V^*$ but instead the coinvariants with respect to the the left action of the division ring $\operatorname{End}_G (V)$ on $V$ and its right action on $V^*$.  The same modification works for the induction functor: for each $f \in \End{V}$, we have two maps $\bigoplus_V V \otimes X \otimes V^* \rightrightarrows \bigoplus_V V \otimes X \otimes V^*$ given by acting on either $V$ or $V^*$, and the correct value of $I(X)$ is the colimit of the collection of these diagrams where $f$ varies over a basis of $\End{V}$.  (This colimit exists because all small colimits exist in an abelian category.)  The half-braidings for $I(X)$ in the split case were written down in \cite[Theorem 2.3]{1004.1533}.  Their formula can be modified to work in the non-split case, by being careful about normalizations and by quotienting out by the action of $\End{V}$ as above.  Indeed, since every object in $Z(\cC)$ is a summand of an induced object, it's possible to prove that the center of a rational form is a rational form of the center explicitly by writing down formulas for the half-braiding on the induced objects following \cite[Theorem 2.3]{1004.1533}.  Here, we choose instead to work directly with the second description of the center.
\end{rem}

\begin{lem} \label{lem:rationalcenter}
Suppose that $\cC_k$ is a rational form of $\cC_K$. The natural functor $Z(\cC_k)\otimes_k K \rightarrow Z(\cC_K)$ is full, faithful, commutes with the induction functors, and is dominant.  Hence,  $Z(\cC_k)$ is a rational form as a braided tensor category of $Z(\cC_K)$.
%If $\cC_k$ is a rational form over $k$ of $\cC_K$ as a fusion category, then $Z(\cC_k)$ is a rational form for $Z(\cC_K)$ as a ribbon fusion category.
\end{lem}

%Recall that $Z(\cC)$ is the dual of the fusion category $\cC \boxtimes \cC^{op}$ over the module category $\cC$ (with the action $(X\boxtimes Y) \otimes M \rightarrow X \otimes M \otimes Y$) \cite{MR1976233}.  Hence Lemma \ref{lem:rationalcenter} follows from the following more general result.

Lemma \ref{lem:rationalcenter} follows immediately from the following more general result.

\begin{lem}
Suppose that $\cD_K$ is a fusion category over $K$, that $\cM_K$ is a semisimple module category over $\cD_K$, that $\cD_k$ is a rational form of $\cD_K$ over $k$ (as a fusion category), and that $\cM_k$ is a rational form of $\cM_K$ over $k$ (as a module category over $\cD_k$).  Let $\cF$ be the full, faithful and dominant functor $\cF: \cD_k \tensor_k K \rightarrow \cD_K$ guaranteed by the fact that $\cD_k$ is a rational form, and $\cG$ the corresponding functor $\cG:\cM_k \tensor_k K \to \cM_K$.  Let $M$ be a simple object in $\cM_k$.  Then, $\cF( \iHom_{\cD_k \actson \cM_k}(M,M))$ and $\iHom_{\cD_K \actson \cM_K}(\cG(M),\cG(M))$ are isomorphic as algebra objects, and there is a  functor $\cF^*: D_k^*\otimes_k K \rightarrow D_K^*$ induced by $\cF$ which is full, faithful and dominant.  Hence, $D_k^*$ is a rational form of $D_K^*$.
\end{lem}
\begin{proof}
Let $A_k = \iHom_{\cD_k \actson \cM_k}(M,M)$ in $\cD_k$ and $A_K =\iHom_{\cD_K \actson \cM_K}(\cG(M),\cG(M))$ in $\cD_K$.  First we show that $\cF(A_k) \cong A_K$ as algebra objects.  We need only check that $\cF(A_k)$ represents the functor  $\cD_K \to \mathrm{Set}$ given by $X \mapsto \Hom{}{X \otimes \cG(M)}{\cG(M)}$, which follows immediately from the fact that $\cM_k$ is a rational form for $\cM_K$.

Since $\cF(A_k) \cong A_K$, the functor $\cF: D_k \otimes_k K \rightarrow D_K$ induces a functor $\cF^*: D_k^* \otimes_k K \rightarrow D_K^*$, which we now describe.  Recall we can interpret $\cD^*_k$ as the $A_k$--$A_k$ bimodules in $\cD_k$ (and similarly for $\cD_K$). Given an $A_k$--$A_k$ bimodule in $\cD_k$, we can construct an $A_K$--$A_K$ bimodule in $\cD_K$ simply by applying the functor $\cF$ to the underlying object, as well as to the morphisms making it a bimodule. We now want to prove that $\cF^*$ is full and faithful.  Suppose that $X$ and $Y$ are $A_k$--$A_k$ bimodules, then among all maps from $X$ to $Y$ the bimodule maps are picked out by satisfying certain linear equations with coefficients in $k$.  Similarly among all maps from $\cF(X)$ to $\cF(Y)$ the bimodule maps are picked out by satisfying the same linear equations with coefficients in $k$.  Thus, since $\cF: D_k \otimes_k K \rightarrow D_K$ is an isomorphism on morphism spaces, it follows that $\cF^*: D_k^* \otimes_k K \rightarrow D_K^*$ is as well.

Thus we need only check that $\cF^*: D_k^* \otimes_k K \rightarrow D_K^*$ is dominant (i.e. every object in the target is a summand of an object in the image of the functor).  Since $\cF(A_k) \cong A_K$ we see that $\cF \circ I$ and $I \circ \cF$ are naturally isomorphic functors. If  $X$ is an arbitrary $A_K$--$A_K$ bimodule, then certainly $X$ is a summand of $I(R(X))$ (where $R$ is the restriction functor which forgets the left action of $A_K$).  Since $\cM_k$ is a rational form of $\cM_K$, we see that $R(X)$ is a summand of $\cF(N)$ for some right $A_k$ module $N$.  Hence, $X$ is a summand of $I( \cF(N)) \iso \cF(I(N)$, and the functor is dominant.

%Let $\underline{\mathrm{Hom}}$ be the internal hom from $\cM_k \times \cM_k \rightarrow \mathrm{Fun}_{\cC_k}(\cM_k, \cM_k)$.  If $F$ is an object in $\mathrm{Fun}_{\cC_K}(\cM_K, \cM_K)$ and $M$ is an object in $\cM_k$, then $F$ is a summand of $\underline{\mathrm{Hom}}(F(X), X)$.
\end{proof}

%We want to prove that $Z(\cH_0^p(\Rational(\zeta_{39})))$ is a complete rational form of $Z(\cH_0^p(\Complex)) \cong Z(\cH_0^d(\Complex)).$
%Note that if $\cC_k$ is a rational form of $\cC_K$ then $Z(\cC_K)$ is a rational form of $Z(\cC)$.  This follows from using the theory of weak Hopf algebras to reduce to considering Drinfel'd double of a weak Hopf algebra \cite{MR1976459}.  Hence the only thing to prove is that the Drinfel'd center $Z(\cH_0^p(\Rational(\zeta_{39})))$ is split (even though $\cH_0^p(\Rational(\zeta_{39}))$ itself is nonsplit).

The center $Z(\cH_0^p(\Complex))$ has previously been described in \cite{MR1832764} (see \cite{MR2468378} for further details). The simple objects are $1, \pi_1, \pi_2, \mu_1, \ldots, \mu_6, \sigma_0, \sigma_1, \sigma_2$.  Let $I$ be the induction functor $I: \cH_0^p \rightarrow Z(\cH_0^p)$. Over $\Complex$, it is described on the level of objects by the graph in Figure \ref{fig:izumi-induction}. Note that over $\Complex$ we have $$I(\JW{2}) \cong \pi_1 \oplus \pi_2 \oplus \mu_1 \oplus \ldots \oplus \mu_6 \oplus \sigma_0 \oplus \sigma_1 \oplus \sigma_2,$$ and in particular there is exactly one copy of each simple except $1$.

\newcommand{\fig}[3]{%
	\begin{figure}[!htb]
	#3
	\caption{#2}
	\label{fig:#1}	
	\end{figure}
}

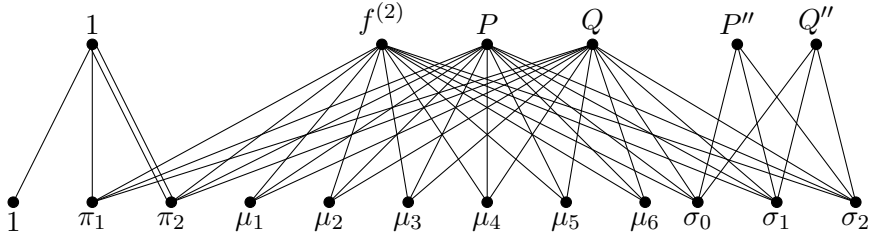
\fig{izumi-induction}{The induction functor for $\cH_0^p(\Complex)$, reproduced from Figure 5 of  \cite{MR1832764}, with $\widehat{{}_0\rho}, \widehat{{}_1\rho}, \widehat{{}_2\rho}, \widehat{\alpha_1}$ and $\widehat{\alpha_2}$ renamed to $\JW{2}, P, Q, P''$ and $Q''$  respectively.}{$$%
%\beginpgfgraphicnamed{\pathtotrunk diagrams/tikz/#1-external}%
\begin{tikzpicture}[baseline=-1mm, scale=.7]
%	\draw (0,0) -- (7,0);
%	\draw (7,0)--(7.7,.7)--(9.7,.7);
%	\draw (7,0)--(7.7,-.7)--(9.7,-.7);

	\filldraw              (0,0) node (1) [below] {$1$} circle (1mm);
	\filldraw              (1.5,0) node (p1)  [below] {$\pi_1$} circle (1mm);
	\filldraw              (3,0) node (p2)  [below] {$\pi_2$} circle (1mm);
\foreach \i in {1,...,6}
	\filldraw              (3+1.5*\i,0) node (m\i) [below] {$\mu_\i$} circle (1mm);

\foreach \i in {0,...,2}
	{\filldraw              (1.5*\i+13, 0) node (s\i) [below] {$\sigma_{\i}$} circle (1mm);}

	\filldraw              (1.5,3) node (1')  [above] {$1$} circle (1mm);
	\filldraw              (7,3) node (f2) [above] {$\JW{2}$} circle (1mm);
	\filldraw              (9,3) node (P) [above] {$P$} circle (1mm);
	\filldraw              (11,3) node (Q) [above] {$Q$} circle (1mm);
	\filldraw              (13.75,3) node (P'') [above] {$P''$} circle (1mm);
	\filldraw              (15.25,3) node (Q'') [above] {$Q''$} circle (1mm);
	
	\draw (1.north)--(1'.south);
	\draw (p1.north)--(1'.south);
\foreach \i in {1,2}
{
	\draw (p\i.north)--(f2.south);
	\draw (p\i.north)--(P.south);
	\draw (p\i.north)--(Q.south);
}
	\draw (p2.80)--(1'.280);
	\draw (p2.100)--(1'.260);
\foreach \i in {1, ..., 6}
	{
	\draw (m\i.north)--(f2.south);
	\draw (m\i.north)--(P.south);
	\draw (m\i.north)--(Q.south);
	}

\foreach \i in {0, ..., 2}
	{
	\draw (s\i.north)--(f2.south);
	\draw (s\i.north)--(P.south);
	\draw (s\i.north)--(Q.south);
	\draw (s\i.north)--(P''.south);
	\draw (s\i.north)--(Q''.south);
	}

\end{tikzpicture}%
%\endpgfgraphicnamed
$$}

%\fig{rational-izumi-induction}{The induction functor for the rational form of $\cH_0^p$.}{$$\inputtikz{rationalinductionfunctor}$$}

\begin{thm}
$Z(\cH_0^p(\Rational(\zeta_{39})))$ is a split fusion category over $\Rational(\zeta_{39})$.
\end{thm}
\begin{proof}
%We want to show that the object $I(\JW{2})$ has the same decomposition over $\Rational(\zeta_{39})$.  
Since $\JW{2}$ is an object in $\cH_0^p(\Rational(\zeta_{39}))$, we have that $I(\JW{2})$ is an object in $Z(\cH_0^p(\Rational(\zeta_{39})))$.  We build explicit projections  on to many of its summands.
Let $\theta_{I(\JW{2})}$ be the ribbon element acting on $I(\JW{2})$.  Over $\Complex$, the ribbon element $\theta_{I(\JW{2})}$ acts on each simple object by the corresponding entry in the $T$-matrix which is diagonal with entries $$(T_{jj}) = (1,1,1, \zeta_{39}^6, \zeta_{39}^{-6}, \zeta_{39}^{15},  \zeta_{39}^{-15}, \zeta_{39}^{18}, \zeta_{39}^{-18}, 1, \zeta_{39}^{13}, \zeta_{39}^{-13}).$$

Consider $\frac{1}{39}\sum_i \zeta_{39}^{m i} \theta_{I(\JW{2})}^i$ for different values of $m$. These give projections onto the eigenspaces of the $T$-matrix Using the values of the $T$-matrix it follows that for these give projections onto each of $\mu_i$, $\sigma_1$, and $\sigma_2$.  Hence all of those simples are defined over $\Rational(\zeta_{39})$.

The sum $\sum_i \theta_{I(\JW{2})}^i$ gives a projection onto $\pi_1 \oplus \pi_2 \oplus \sigma_0$.  Let $F$ be a single strand labelled by $\pi_1 \oplus \pi_2 \oplus \sigma_0$ with a ring around it labelled by $\sigma_1$.  Looking at the corresponding entries in the $S$-matrix (described in \cite{MR2468378}) we see that $\frac{5 + \sqrt{13}}{18} (\id + F)$ 
 gives a projection onto $\pi_1$ while its complement gives a projection onto $\pi_2 \oplus \sigma_0$.

Consider $I(1) \cong 1\oplus \pi_1 \oplus 2 \pi_2$.  Notice that here $\pi_2$ is a summand while $\sigma_0$ is not.  In particular, there is a map over $\Complex$ from $I(1)$ to $\pi_2 \oplus \sigma_0$, and hence over $\Rational(\zeta_{39})$ there also must be such a map.  Hence by semisimplicity, we must have $\sigma_0$ and $\pi_2$ are objects over $\Rational(\zeta_{39})$.
\end{proof}

\section{Galois conjugates}
\label{sec:galois}
%!TEX root = ../article.tex

The field $$K = \Rational(\lambda_0) = \Rational\left(i\sqrt{\frac{-1+\sqrt{13}}{2}}\right)$$ is a degree $4$ non-Galois extension of $\Rational$.  Its Galois closure $L$ has degree $8$ and has Galois group the dihedral group with $8$ elements which we think of as the automorphisms of a fixed square.  The field $K$ is the fixed points of the subgroup generated by a reflection in one of the sides in the square.  In particular, the orbit of $K$ under the Galois group $\text{Gal}(L/\Rational)$ consists of two fields, $K$ and $$K' = \Rational\left(i\sqrt{\frac{-1-\sqrt{13}}{2}}\right).$$  Each of $K$ and $K'$ is fixed pointwise by two elements of $\text{Gal}(L/\Rational)$ and each has a single nontrivial automorphism over $\Rational$.  The field automorphism of $K/\Rational$ acts on $\cH_0^p$ by the diagram automorphism which interchanges $P$ and $Q$.  In particular, using the Galois action we can only construct a single new fusion category which is defined over $K'$.  This category is non-unitary and the dimensions of the objects are given by replacing $\sqrt{13}$ with $-\sqrt{13}$ everywhere.

The story is similar for $ \Rational(\lambda_1)$ and $\cH_1^p$. Again, there is a single nontrivial automorphism of the field which acts on the fusion category by diagram automorphism.  There are two Galois conjugate fusion categories which are non-unitary and whose dimensions are given by the action of $\text{Gal}(\Rational(d_1^2)/\Rational)$ on the old dimensions.

% ----------------------------------------------------------------
\hfuzz5pt 
\newcommand{\urlprefix}{}
\bibliographystyle{gtart}
%Included for winedt:
%input "bibliography/bibliography.bib"
\bibliography{bibliography/bibliography}
% ----------------------------------------------------------------

This paper is available online at \arxiv{1002.0168}, and at
\url{http://tqft.net/noncyclotomic}.

% A GTART necessity:
% \Addresses
% ----------------------------------------------------------------
\end{document}